\documentclass[11pt,a4paper,openany]{article}
\usepackage[utf8]{inputenc}
\usepackage{graphicx}
\usepackage{graphics}
\usepackage{amscd,amsmath,amstext,amsfonts,amsbsy,amssymb,amsthm,eufrak}
\usepackage{mathrsfs}
\usepackage[unicode,colorlinks,linkcolor = blue,citecolor = red]{hyperref}
\usepackage{nccmath}
\usepackage{framed}
\usepackage{authblk}
\usepackage{hyperref, xcolor}
\usepackage{amsthm}
\usepackage{subcaption}

\usepackage{multicol,graphicx,framed,enumerate}
\usepackage{array,multirow,longtable,fancyhdr,booktabs}
\usepackage{makeidx}
\usepackage{afterpage}
\usepackage{appendix}

\graphicspath{ {./numerical results/} }

\newtheorem{thm}{Theorem}[section]
\newtheorem{lem}[thm]{Lemma}
\newtheorem{proposition}[thm]{Proposition}

\newtheorem{assumption}[thm]{Assumption}
\newtheorem{definition}[thm]{Definition}
\newtheorem*{definition*}{Definition}
\newtheorem*{thm*}{Theorem}

\newtheorem{remark}[thm]{Remark}

\newcommand{\bequ}{\begin{equation}}
\newcommand{\eequ}{\end{equation}}
\def\bn{\begin{eqnarray*}}\def\en{\end{eqnarray*}}
\def\bu{\begin{equation}}\def\eu{\end{equation}}

\newcommand{\card}{\mbox{Card}}

\newcommand{\Var}{\mbox{Var}}
\newcommand{\Atan}{\mbox{atan2}}

\def\Z{\mathbb{Z}}
\def\N{\mathbb{N}}

\def\R{\mathbb{R}}
\def\L{\mathbb{L}}

\def\E{\mathbb{E}}

\def\S{\mathbb{S}}
\def\r{{r}}

\def \compi{\mathrm{i}}

\makeatother

\author{Tien Dat Nguyen \thanks{Faculty of Mathematics and Computer Science, Vietnam National University Ho Chi Minh city, Vietnam; \texttt{E-mail: ndat@hcmus.edu.vn}}
\thanks{Vietnam National University Ho Chi Minh city, Vietnam},
Thanh Mai Pham Ngoc\thanks{LAGA, UMR 7539, Institut Galil\'ee, Universit\'e Sorbonne Paris Nord, 93430,
Villetaneuse, France; \texttt{E-mail: phamngoc@math.univ-paris13.fr}}}

\title{Adaptive  estimation for nonparametric circular regression with errors in variables}
\date{\today}
\begin{document}
\maketitle

\begin{abstract}
This paper investigates the nonparametric estimation of a circular regression function in an errors-in-variables framework. Two settings are studied, depending on whether the covariates are circular or linear. Adaptive estimators are constructed and their theoretical performance is assessed through convergence rates over Sobolev and H\"older smoothness classes. Numerical experiments on simulated and real datasets illustrate the practical relevance of the methodology.
\end{abstract}


\section{Introduction}

We consider the problem of circular regression with errors in variables. We observe the i.i.d. dataset $(Z_1, \Theta_1), \dots, (Z_n,\Theta_n)  $ where
$$
\Theta_i = m(X_i) + \zeta_i \quad  \quad  \textrm{(mod $2\pi$)}.
$$
 and 
$$
Z_i = X_i + \varepsilon_i,
$$
with the responses $\Theta_i$ and the regression errors $\zeta_i$ lying in $\mathbb{S}^1$, the unit circle in $\mathbb{R}^2$.

We consider two scenarios: the first where the covariates $X_i$ belong to $\mathbb{S}^1$ (circular predictor case), and the second where the $X_i$ lie in $[0,1]$ (linear predictor case). The predictors $X_i$ are latent, and the covariate errors $\varepsilon_i$ are i.i.d. unobservable random variables with density $f_\varepsilon$, whose characteristic function is assumed to be known. 
The $\varepsilon_i$ are independent of $X_i$. Our objective is to estimate the circular regression function $m$ nonparametrically from the i.i.d. dataset $(Z_1, \Theta_1), \dots, (Z_n,\Theta_n)$. The term circular data is used to distinguish such data from those supported on the real line (or its subsets), henceforth referred to as linear data.

Circular or angular data are encountered in various scientific fields, such as biology (e.g., directions of animal migration), bioinformatics (e.g., protein conformational angles), geology (e.g., rock fracture orientations), medicine (e.g., circadian rhythms), forensics (e.g., crime timing), and the social sciences (e.g., time-of-day or calendar effects). Comprehensive surveys of statistical methods for circular data can be found in Mardia and Jupp \cite{book:Mardia-Jupp}, Jammalamadaka and SenGupta \cite{book:Jammalamadaka-SenGupta}, Ley and Verdebout \cite{book:Ley-Verdebout}, and recent advances are compiled in Pewsey and Garc\'ia-Portugu\'es \cite{Pewsey-GarciaPortuge}.

The circular response in our model introduces substantial statistical challenges compared to classical regression with linear responses. Circular regression necessitates methodologies that respect the geometry of the circle. In the parametric setting, various ad hoc approaches have been proposed (see \cite{Gould}, \cite{Johnson-Wehrly-1}, \cite{Johnson-Wehrly-2}, \cite{Fisher-Lee:1992}, \cite{Kato-Shimizu-Shieh}, \cite{Downs-Mardia}, \cite{Fernandez-Duran}). To the best of our knowledge, even without measurement errors, nonparametric circular regression has been explored in only a few works, such as \cite{Marzio-Panzera-Taylor}, \cite{Nguyen-PhamNgoc-Rivoirard} and \cite{Jeon-Park-Keilegom2021}. Only one recent study \cite{Di-Marzio-Fensore-Taylor} addresses the case involving errors in variables. Errors-in-variables models arise in many experimental contexts where data are corrupted by measurement errors. These models have been extensively studied on the real line (see \cite{Meister}, \cite{Comte-Taupin}, \cite{Koo-Lee}, \cite{Delaigle-Hall-Jamshidi}, \cite{Carroll-Delaigle-Hall}, \cite{Chesneau}, \cite{Jeon-Park-Keilegom2022}, \cite{Jeon-Park-Keilegom2024}), but much less so in the circular response setting. Such models are particularly demanding as they involve a deconvolution step in the estimation process.

As previously mentioned, the only work addressing nonparametric circular regression with measurement errors is \cite{Di-Marzio-Fensore-Taylor}, hence this problem is largely unexplored in statistic. In that study, the authors propose a deconvolution kernel estimator for both the circular and linear covariate cases. Their approach depends critically on the choice of key parameters: a kernel  bandwidth (related to the regression estimation) and a spectral cutoff (related to deconvolution). However, the selection of these parameters is not theoretically addressed, and no convergence rates or adaptive procedures are provided. By adaptation, we mean that the estimator does not require prior knowledge of the regularity of the regression function and can automatically adjust to it--a feature essential in practice. In their numerical implementation, the authors of \cite{Di-Marzio-Fensore-Taylor} rely on cross-validation to select the kernel bandwidth and choose arbitrary spectral cutoff values. In contrast, we propose a novel data-driven estimation procedure for the circular regression function in both settings: circular and linear predictors. We introduce adaptive rules for bandwidth selection inspired from the Goldenshluger and Lepski method (see \cite{Goldenshluger-Lepski:2011}) and deconvolution density problems (see \cite{Comte-Lacour}) and establish convergence rates for the pointwise risk over Sobolev and H\"older classes. Our methodology is based on concentration inequalities. Finally, we assess the performance of our estimators through both simulation studies and real data applications. We show that our statistical procedure achieves satisfactory performance.
\medskip

The paper is organized as follows. Section~\ref{sec:framework} details the framework and the specific features of circular regression. Sections~\ref{section-circular} and~\ref{section-linear} deal respectively with the circular and linear predictor cases. Numerical experiments are presented in Section~\ref{numerical-results}, and all proofs are collected in Section~\ref{proofs}.

\medskip

 \textit{Notations.} It is necessary to equip the reader with some notations. 
 \medskip
 
 The notation $x_+$ means $\max(0,x).$ For two integers $a,b$ we denote $a \wedge b := \min(a,b)$. 
For two functions $u,v$ we denote $u(x) \precsim v(x) $ if there exists a positive constant $C$ not depending on $x$ such that $u(x) \leq C v(x)$ and $u(x) \approx v(x)$ if $u(x) \precsim v(x)$ and $v(x) \precsim u(x)$. 

\medskip

In Section \ref{section-circular}, $\| \cdot \|_{\ell_p}, \, p \geq 1$ denotes the norm on complex sequences space
$$
\| x \|_{\ell_p} := \left (\sum_{l \in \mathbb{Z}} |x_l|^p \right )^{1/p},
$$
and the $l$-th Fourier coefficient of a function $f \in \mathbb{L}^1(\S^1)$ is defined as  $f^{\star}(l) := \int_{\S^1} f(x) e^{\mathrm{i}lx} dx, l \in \mathbb{Z}$.
\medskip

In Section \ref{section-linear}, $\left\| \cdot \right\|_{\L^{1}(\R)}$ and $\left\| \cdot \right\|_{\L^{2}(\R)}$ respectively denote the $\L^{1}$ and $\L^{2}$ norm on $\R$ with respect to the Lebesgue measure:
\begin{equation*}
\left\| f \right\|_{\L^{1}(\R)}  =  \int_{\R} |f(y)| dy , \quad \left\| f \right\|_{\L^{2}(\R)}  = \Big( \int_{\R} |f(y)|^{2} dy  \Big)^{1/2},
\end{equation*}
and the $\L^{\infty}$ norm is defined by $\left\|f\right\|_{\infty} = \sup_{y \in \R} |f(y)|$. And for $f \in \L^{1}(\R)$,  \ $f^{\star}(x):= \int e^{\compi xt} f(t) dt, x \in \R$ denotes the Fourier transform of $f$.  Moreover, we denote $*$ the classical convolution product defined for functions $f, g$ by $f * g(x) := \int_{\R} f(x-y)g(y) dy $, $x \in \R$. 
\medskip
 
 In the sequel, a point on $\mathbb{S}^1$ will not be represented as a two-dimensional vector $\bold{w}=(w_2,w_1)^\top$ with unit Euclidean norm but as an angle $\theta=atan2(w_1,w_2) $ where
\begin{definition}\label{def:full.formula.atan2} The function $\Atan: \R^2\setminus (0,0)\mapsto [-\pi;\pi]$ is defined for any $(w_1,w_2)\in\R^2\setminus (0,0)$ by \emph{
		$$
		\Atan (w_1,w_2):= \left \{
		\begin{array}{ll}
		\arctan\big( \frac{w_1}{w_2} \big) & \quad \mbox{if } \quad w_2 \geq 0 , w_{1} \neq 0 \\
		{0} & \quad {\mbox{if } \quad w_2 > 0 , w_{1} = 0} \\
		\arctan\big( \frac{w_1}{w_2} \big) + \pi & \quad \mbox{if } \quad w_2 < 0, w_1 >0 \\
		\arctan\big( \frac{w_1}{w_2} \big) - \pi & \quad \mbox{if } \quad w_2 < 0, w_1 \leq 0,
		\end{array}
		\right.
		$$
		with $\arctan$ taking values in $[- \pi/ 2, \pi/ 2]$. In particular for $w_1>0$, $\Atan (w_1,0)=\arctan(+\infty)=\pi/2$ and $\Atan (-w_1,0)=\arctan(-\infty)=-\pi/2$.
}\end{definition} 
In this definition, one has arbitrarily fixed the origin of $\mathbb{S}^1$ at $(1,0)^\top$ and uses the anti-clockwise direction as positive. Thus, a circular random variable can be represented as angle over $[-\pi, \pi)$. Observe that $\Atan (0,0)$ is not defined.

\section{Regression with a circular response}\label{sec:framework}
Circular data are fundamentally different from linear data due to their periodicity, and thus require specific techniques.  To measure the closeness between two angles $\theta_1$ and $\theta_2$, we do not consider the natural distance
$$d(\theta_{1} , \theta_{2}) := \min \left\{ \big| \theta_{1}  - \theta_{2}  +  2k\pi \big| : k \in \mathbb{Z} \right\},\quad \theta_1, \theta_2 \in [-\pi, \pi),$$
but focus on $d_c$ with $$d_{c} (\theta_{1}, \theta_{2}) := 1 - \cos(\theta_{1} - \theta_{2}),\quad \theta_1, \theta_2 \in [-\pi, \pi),$$   which is extensively used in the literature of directional statistics (see for instance Section 2 in the seminal monograph by Mardia and Jupp \cite{book:Mardia-Jupp}, Section 3.2.1 of \cite{book:Ley-Verdebout}, \cite{Marzio-Panzera-Taylor} or \cite{MVila-FFernandez-Crujeiras-Panzera}). Note that the divergence $d_c$ corresponds to the usual squared Euclidean norm in $\R^{2}$. Indeed, the angles $\theta_{1}$ and $\theta_{2}$ determine the corresponding points $( \cos \theta_{1} ,  \sin \theta_{1})$ and $( \cos \theta_{2} ,  \sin \theta_{2})$ respectively on the unit circle $\mathbb{S}^{1}$.  Then, the usual squared Euclidean norm  in $\R^{2}$ reads
\begin{align*}
\big( \cos \theta_{1} - \cos \theta_{2} \big)^{2} + \big( \sin \theta_{1} - \sin \theta_{2} \big)^{2} = 2. \big[ 1 - \cos (\theta_{1} - \theta_{2}) \big] = 2. d_{c}(\theta_{1}, \theta_{2}).
\end{align*}
Hence, $\sqrt{d_c}$ is a distance on $[-\pi,\pi)$ and we naturally look for a measurable function $m$ such that: 
\begin{align}\label{minimisation}
\mathbb{E} \big[ d_c(\Theta, m(X))  \big] =
\underset{f: \; \Omega \rightarrow [-\pi,\pi) }{ \textrm{min} } \; \mathbb{E} \big[ d_c(\Theta, f(X))  \big],
\end{align}
where $\Omega$ denotes either $[0,1]$ or $\mathbb{S}^1$ depending on the nature of the predictors $X_i$ and the minimum is taken over $[-\pi, \pi)$-valued functions $f$ that are measurable with respect to the $\sigma$-algebra generated by $X$. It is interesting to notice that the minimization problem (\ref{minimisation}) is directly linked to the definition of the Frechet mean on the circle (see \cite{Charlier:2013}). 
Furthermore, in directional statistics, the problem of finding such a regression function $m(X)$ as defined in (\ref{minimisation}) has been already considered (see \cite{Nguyen-PhamNgoc-Rivoirard} and \cite{MVila-FFernandez-Crujeiras-Panzera}).
\\
Now let us work conditionally to $X$. For $x\in\Omega$,  let
\begin{equation}\label{def:m1m2}
m_{1}(x) := \mathbb{E} \big( \sin (\Theta) | X=x\big)\quad and \quad m_{2}(x) := \mathbb{E} \big( \cos (\Theta) | X=x\big).
\end{equation}
Moreover, write for an arbitrary function $f: \Omega \rightarrow [-\pi,\pi) $
\begin{align*}
\mathbb{E} \big[ \cos (\Theta - f(X)) | X \big] &= \cos(f(X))m_{2}(X) + \sin(f(X))m_{1}(X)
\\
&=  \sqrt{(m_{2}(X))^{2} + (m_{1}(X))^{2} } \cos (f(X) - \gamma(X)) ,
\end{align*}
where $\gamma : \;\Omega \rightarrow [-\pi,\pi)$ is defined for $x \in \Omega$
\begin{equation*}
\cos(\gamma(x)) := \dfrac{m_{2}(x)}{\sqrt{(m_{2}(x))^{2} + (m_{1}(x))^{2} }}   \quad  \textrm{ and } \quad  \sin(\gamma(x)) := \dfrac{m_{1}(x)}{\sqrt{(m_{2}(x))^{2} + (m_{1}(x))^{2}} }.
\end{equation*}
Observe that 
$$\gamma(x) = \Atan ( m_{1}(x),m_{2}(x)).$$
Thus, we have
\begin{align*} \label{eq:m(X).given.by.atan2}
\underset{f: \;\Omega \rightarrow [-\pi,\pi) }{ \textrm{min} }  \; \mathbb{E} \big[ d_c(\Theta, f(X))  \big]
&=1- \underset{f: \; \Omega \rightarrow [-\pi,\pi) }{ \textrm{max} } \E\Big[\mathbb{E} \big[ \cos (\Theta - f(X)) | X \big] \Big]\\
&=1- \underset{f: \;\Omega \rightarrow [-\pi,\pi) }{ \textrm{max} } \E\Big[\sqrt{m_1^2(X)+m_2^2(X)}\cos (f(X) - \gamma(X))\Big].  
\end{align*}
Finally the minimizer of the minimization problem \eqref{minimisation} is achieved for 
$$f(x)=\gamma(x) = \Atan \big( m_{1}(x),m_{2}(x)\big).$$

In conclusion, the circular nature of the response is taken into account by the arctangent of the ratio of the conditional expectation of sine and cosine components of $\Theta$ given $X$.
Accordingly, we tackle the regression problem by estimating the function \begin{equation}
m(x) = \Atan (m_1(x), m_2(x)),\quad x\in \Omega,
\label{formula:regress-function.m}
\end{equation} 
with $m_1$ and $m_2$ defined in \eqref{def:m1m2}.

\begin{remark}
Observe that if $m_1(x)=m_2(x)=0$, then $m(x)$ is not defined. This occurs if and only if
$$
\phi_1(f(\cdot|x)):=\int_{-\pi}^{\pi} e^{\mathrm{i}\theta}f(\theta|x) d\theta=0,
$$
where $f(\cdot|x)$ denotes the conditional density of $\Theta|X=x$. Note that $\phi_1(f(\cdot|x))$ plays a specific role in the literature of directional statistics. See for instance Section 3.4.2 of \cite{book:Mardia-Jupp}.
\end{remark}
In the sequel, we estimate the circular regression function $m$ as defined in \eqref{formula:regress-function.m} under the condition
\begin{equation}\label{not0}
\phi_1(f(\cdot|x))\not=0.
\end{equation}

Since the function $\Atan (w_{1},w_{2})$ is undefined when $w_{1} = w_{2} = 0$, it is reasonable to consider the following assumption:
\begin{assumption}\label{assumption:c_{low}} 
		\begin{equation*}
			m_1(x)\not= 0 \quad \textrm{ or } \quad m_2(x)\not= 0 , \quad \textrm{ for } x \in \Omega,  
		\end{equation*}
		and let $\delta > 0$ be defined as
		\begin{equation}\label{deltawx}
			\delta=\left\{
			\begin{array}{ccc}
				\min\big(|m_1(x)|, |m_2(x)|\big)  &\mbox{ if }   & m_1(x)\not= 0 \mbox{ and }  m_2(x)\not= 0, 
				\\[0.2cm]  
				|m_1(x)|  &\mbox{ if }   & m_2(x)= 0,  \\[0.2cm]  
				|m_2(x)|  &\mbox{ if }   & m_1(x)= 0.
			\end{array}
			\right.
			\end{equation}
\end{assumption}

\begin{assumption}
Assume that
there exists $\delta_{X} > 0$ such that
\begin{align}  
f_{X}(x) \geq \delta_{X} , \hspace{0.2cm} \textrm{ for all } \hspace{0.1cm} x \in \Omega .  
\label{assumption:assumption.on.f_X}
\end{align} 
\end{assumption}

Now let us expose the overall estimation strategy.  In regression with random design, the classical idea to estimate the regression functions $m_i$ with $i \in \{1,2\}$ is to write $m_i$ as a ratio $\frac{m_{i} f_{X}}{f_X}$, with $f_X$ the density of $X$. Hence one usually considers estimators of intermediate functions $p_i$ defined as:
$$p_{i} :=  m_{i}    f_{X}, \quad i \in \{1,2\}, $$
along with an estimator of $f_{X}$ obtained by a deconvolution step due to the errors-in-variables setting.  In constrat, here in our circular response setting, the estimator of $f_X$ does not have any impact because of the ratio involved in the $\Atan$ function (see Definition \ref{def:full.formula.atan2}) which entails that $\Atan(m_1, m_2)= \Atan(p_1,p_2)$. This is actually a specificity of the circular response regression compared with the linear one. Accordingly,  our strategy will consist in, first estimating $p_1$ and $p_2$ and then $m=\Atan(p_1,p_2)$.
 
 \section{Circular predictor}\label{section-circular}
 
 \subsection{Model}
We begin by considering the case where the predictor is circular, which represents the most challenging and novel setting. More specifically, we study the following  model:

$$
\Theta= m(X) + \zeta   \quad  \textrm{(mod $2\pi$)},
$$
with $\Theta, X$ and $ \zeta \in \mathbb{S}^1$. The $X$ are latent and we only have access to:
$$
Z = X + \varepsilon \quad  \textrm{(mod $2\pi$)}.
$$
 We assume that  $X$ and $\varepsilon$ are independent and  the characteristic function of the density $f_\varepsilon$ is known.  From an i.i.d dataset $(Z_1, \Theta_1), \dots, (Z_n,\Theta_n)  $, we aim to estimate $m: \mathbb{S}^1 \mapsto \mathbb{S}^1$.

\subsection{Assumptions}
We shall present our results in terms of Sobolev classes. Due to circular covariates, we deal with Fourier series instead of Fourier transforms. For any function  $f(x)= \frac{1}{2\pi} \sum_{l \in \mathbb{Z}} f^{\star}(l) e^{-\compi l x}$ in $ \mathbb{L}^2(\S^1)$ let 
$$
\| f\|^2_{{W}_\beta} := \sum_{l \in \Z} |f^\star(l)|^2 (1+l^2)^\beta.
$$
We define the class $\mathcal{W}(\beta, S)$ to be 
\begin{equation}
\mathcal{W}(\beta, S)= \{ f  \in  \mathbb{L}^2(\S^1): \; \| f\|^2_{W_\beta} \leq S^2 \}.
\end{equation}


Following \cite{Di-Marzio-Fensore-Panzera-Taylor:2020}, we make the following assumptions about the noise in the variables.
\begin{assumption}
The noise $\varepsilon$ is called ordinary smooth  (OS) if there exist $\nu, a_0, a_1 >0$, such that
\begin{equation}\label{OS-circle}
 a_0( 1+ |l|)^{-\nu} \leq  |f^{\star}_{\varepsilon}(l)| \leq a_1( 1+ |l|)^{-\nu}, \quad l \in \mathbb{Z}.
\end{equation}
The noise $\varepsilon$ is called supersmooth (SS) if there exist $a_0, a_1, b, a >0$ and $ c \in \R$
\begin{equation}\label{SS-circle}
a_0(|l| +1)^{c} \exp(-b |l|^a) \leq  |f^{\star}_{\varepsilon}(l)| \leq a_1(|l| +1)^{c} \exp(-b |l|^a), \quad l \in \mathbb{Z}.
\end{equation}
\end{assumption}

Let us recall that some well-known circular densities belong to these two types of noise. Examples of supersmooth densities include the densities of wrapped Normal and
wrapped Cauchy distributions; conversely, the wrapped Laplace and the wrapped Gamma densities are examples of ordinary smooth ones. 

%
\subsection{Preliminary estimators}

As explained at the end of Section \ref{sec:framework}, the target function $m$ is given by $m=\Atan (m_1,m_2)= \Atan(p_1, p_2)$ with 
$p_1(x)= \E (\sin \Theta_1 | X=x)f_X(x)$ and $p_2(x)= \E (\cos \Theta_1 | X=x)f_X(x)$. We first aim to find  estimates of $p_1$ and $p_2$. 
To this end, we adopt an orthogonal projection method onto the Fourier basis of $\mathbb{L}^2(\mathbb{S}^1)$. 

 In the nonparametric framework, the usual idea is to find an estimator of the $L$-th projection of $p_i, i \in \{1,2\}$ namely an estimator of 
$$
p_{i,L_i}(x):=\frac{1}{2\pi}\sum_{| l| \leq L _i} p_i^\star(l) e^{-\mathrm{i} l x}, \; x \in \S^1.
$$

We propose the following estimates for $p_1$ and $p_2$:

\begin{eqnarray*}
\widehat p_{1,L_1}(x)= \sum_{| l| \leq L_1 }\frac {1}{ 2 \pi n} \sum_{j=1}^n \sin(\Theta_j) \frac{e^{\compi l Z_j}}{f^\star_{\varepsilon}(l)} e^{-\compi l x}, \; x \in \S^1, \\
\widehat p_{2,L_2}(x)= \sum_{| l| \leq L _2}\frac {1}{ 2 \pi n} \sum_{j=1}^n \cos(\Theta_j) \frac{e^{\compi l Z_j}}{f^\star_{\varepsilon}(l)} e^{-\compi l x}, \;x \in \S^1.
\end{eqnarray*}

The expression of $\hat p_{1,L_1}(x)$ (same line of reasoning for $\hat p_{2,L_2}(x)$) is justified by the fact that 
\begin{eqnarray*}
\E  \left [\widehat p_{1,L_1}(x) \right ] &=&  \frac{1}{2\pi}\sum_{| l| \leq L_1 }\frac{e^{-\compi l x}}{f^\star_{\varepsilon}(l)} \E (\E(\sin(\Theta_1)e^{\compi l(X_1+\varepsilon_1)} |X_1))) 
=  \frac{1}{2\pi}\sum_{| l| \leq L_1 }e^{-\mathrm{i} l x} \E(e^{\compi l X_1}m_1(X_1)) \\
&=&  \frac{1}{2\pi}\sum_{| l| \leq L_1 } p_1^\star(l) e^{-\compi l x} = p_{1,L_1}(x),
\end{eqnarray*}
which entails that $\widehat p_{1,L_1}(x)$ is a unbiased estimator of $p_{1,L_1}(x)$.

\subsubsection{Bias and Variance}

We have the following bias-variance decomposition of the pointwise risk of $\hat p_{i,L_i}(x)$: 
$$
\E |\widehat p_{i,L_i}(x) - p_i(x)|^2 =  \underbrace{\E | \widehat p_{i,L_i}(x) -  \E[\widehat p_{i,L_i}(x)] |^2}_{=:Var(\hat p_{i,L_i}(x))} + {| \E[\widehat p_{i,L_i}(x)] -p_i(x)|^2}_{}.
$$

\begin{proposition} \label{var-hat-p} For $i\in \{1,2 \}$, if $p_i \in \mathcal{W}(\beta_i, S_i)$ with $\beta_i> \frac 1 2$,  the bias is controlled as follows
$$| \E[\widehat p_{i,L_i}(x)] -p_i(x)|  = \frac{1}{2 \pi } \left  | \sum_{|l| \geq L_i} p_i^\star(l) \right |  \leq C(S_i,\beta_i) L_i^{-\beta_i},$$
with $C(S_i, \beta_i)$ a constant depending on $S_i$ and $\beta_i$. The variance  is controlled by

$$
Var(\widehat p_{i,L_i}(x))  \leq V_0(n,L_i)$$

with 
$$
 V_0(n,L_i) := \frac{1}{(2\pi)^2n}\min \left \{ \left  ( \sum_{ | l | \leq L_i}  \frac{1}{ |f^\star_{\varepsilon}(l)| }\right )^2,    \| f^\star_{\varepsilon} \|_{\ell_1}  \sum_{|l | \leq 2L_i }   \frac{1}{ |f^\star_{\varepsilon}(l)|^2}    \right\}.
$$
Furthermore, we have that
\begin{enumerate}
\item If the noise $\varepsilon$ is OS (satisfying (\ref{OS-circle})) with $\nu>1$ and provided that $L_i \geq 1$, for $ i \in \{1, 2 \}$, we have 
$$
Var(\widehat p_{i,L_i}(x))  \leq \frac{\max \{1, \|  f^\star_{\varepsilon}\|_{\ell_1} \}}{(2\pi)^2n} L_i^{2\nu +1 }. 
$$

\item If the noise $\varepsilon$ is SS  (satisfying (\ref{SS-circle})), we have 
$$
Var(\widehat p_{i,L_i}(x)) \leq \frac{\max\{1, \|  f^\star_{\varepsilon}\|_{\ell_1}\}}{(2\pi)^2n} L_i^{-2 c+1 } e^{2 b L_i^a}.
$$
\end{enumerate}
\end{proposition}
The proof of Proposition \ref{var-hat-p} is given in Section \ref{preuve-proposition-CC}.

\subsubsection{Rates of convergence} \label{circular-rates}

When the noise $\varepsilon$ is ordinary smooth and the functions $p_i$ belong to a Sobolev class $\mathcal{W}(\beta_i, S_i)$, $i \in \{1, 2\}$, we deduce from Proposition~\ref{var-hat-p}, via the usual bias-variance tradeoff, that the optimal choice of the projection level is given by:
$$L_{i,{opt,OS}} = \arg \min_{L \in \N^*} \left \{ L^{-2\beta_i} + \frac{1}{n} L^{2\nu +1}  \right \} \Rightarrow
 L_{i,{opt,OS}}  \varpropto {n}^{\frac{1}{2\beta_i + 2 \nu +1 }},
 $$
and the rate of convergence for the pointwise risk to estimate $p_i$ is  $n^{-\frac{2\beta_i}{2\beta_i + 2\nu +1}}$. This corresponds to
 the usual minimax rate of convergence for estimating a univariate function of regularity $\beta_i$ in an ordinary ill-posed inverse problem of severity $\nu$ (see Chapter 3 in \cite{Meister}). One remarks that the optimal level $ L_{i,{opt, OS}}$ depends on the regularity $\beta_i$ which is unknown. This problem will be tackle in Section  \ref{adaptation} where we propose a data-driven procedure to select  $L_i$. 
\medskip

\noindent When the noise $\varepsilon$ is supersmooth, and the functions $p_i$ belong to a Sobolev class $\mathcal{W}(\beta_i, S_i)$, again from Proposition \ref{var-hat-p}, we deduce that the optimal levels are given by:
\begin{equation} \label{Lopt-SS}
L_{i, {opt, SS}} = \arg \min_{L \in \N^*} \left \{ L^{-2\beta_i} + \frac{1}{n} L^{-2c +1}  e^{2bL^a }\right \} \Rightarrow
L_{i,{opt, SS}}  \varpropto L_{{opt, SS}} := \left( \frac{\log n }{2b}\right)^{\frac 1 a }.
\end{equation}
 The subsequent rates of convergence for the pointwise risk to estimate  $p_i$ are  $\left( {\log n}\right )^{\frac {-2 \beta_i }{a}  }$.

\medskip


 In the case of a supersmooth noise, unlike in the OS noise setting, the optimal level $L_{opt, SS}$ does not depend on the regularity of the functions $p_i$, but solely on the noise parameters, which are assumed to be known. As a consequence, no level selection procedure is required in this case.

%
%
\subsection{Adaptive estimation}\label{adaptation}
As shown in the previous section, when the noise is ordinary smooth, it is necessary to design a data-driven procedure for selecting the levels $L_1$ and $L_2$. To this end, we employ a Goldenshluger-Lepski selection rule adapted to projection methods and indirect observations (see \cite{Chichignoud-Hoang-PhamNgoc-Rivoirard} and \cite{Comte-Lacour}). Consider the following collection of levels: 
\begin{equation}\label{set-L}
\mathcal{L}= \left \{ L \in \{1, \dots ,  n   \}, \frac{\sum_{|l | \leq 2L }   \frac{1}{ |f^\star_{\varepsilon}(l)|^2} }{\left (\sum_{ | l | \leq L}  \frac{1}{ |f^\star_{\varepsilon}(l)| } \right )^2}  \geq \frac{\log n}{n } \right  \}. 
\end{equation}
Note that condition expressed in $(\ref{set-L})$ amounts to have $L  \precsim  \frac{n}{\log n}$. 

For $i \in \{1, 2 \}$, let us set
\begin{equation}\label{def-A-li}
A(L_i, x)=\sup_{L_i'\in \mathcal L} \left  \{| \widehat p_{i, L_i'}(x) - \widehat p_{i, L_i \wedge L_i'}(x)| -\sqrt{V(n, L_i') } \right \}_{+},
\end{equation}
with
$$
 V(n, L_i)= c_{0, i} \log(n) V_0(n, L_i),
$$
and $c_{0, i}>0$ is a tuning constant which will be specified later  while $V_0(n,L_i)$ has been defined in Proposition \ref{var-hat-p}. 
\medskip

Finally, we define our selected levels as 
 \begin{equation}
\hat L_i  := \arg \min_{L _i\in \mathcal L} \left  \{ A(L_i, x) + \sqrt{V(n, L_i)}  \right \}.
\label{equa:adaptive.selection.L1-L2:form}
\end{equation}

\bigskip

We now obtain the following theorem which states an oracle inequality.

\begin{thm}\label{theorem-proba-oracle}

Suppose that the noise $\varepsilon$ is OS (satisfying (\ref{OS-circle})).  Let $q\geq 1$ and 
$
\min \{c_{0, 1}, c_{0,2}\} \geq \frac{64(2 +q)^2 }{\min(1, \| f_{\varepsilon}^\star\|_{\ell_1})}.
$
Then with probability larger than $1-4n^{-q}$ we have 

\begin{equation} \label{proba-oracle}
| \widehat p_{i, \hat L_i} (x) -p_i(x)| \leq \inf_{L_i \in \mathcal L} \left \{ \frac{1}{ \pi}  \sum_{|l| \geq L_i} |p_i^\star(l)| + \frac 5 2 \sqrt{V(n,L_i) } \right \}, \quad i \in \{1, 2\}.
\end{equation}

\end{thm}

The proof of Theorem \ref{theorem-proba-oracle} is given in Section \ref{preuve-theorem-proba-oracle}.

Theorem \ref{theorem-proba-oracle} highlights the bias-variance decomposition of the pointwise risk. The first term  of the r.h.s is a bias term (see Proposition \ref{var-hat-p}) while the second term is a variance term.

\subsubsection{Rates of convergence for  $m$}
In this section, we derive rates of convergence for estimating the target regression function $m$.

\medskip

\noindent In the case of an ordinary smooth noise, we set the final estimator of $m$ to be:
\begin{equation}\label{estimator-m-circular}
\widehat m_{\hat L}(x) := \Atan(\widehat p_{1, \hat L_1}(x), \widehat p_{2, \hat L_2}(x)).
\end{equation}
\noindent In the case of a supersmooth  noise, remind that no selection level is required (see  Section \ref{circular-rates}) and we have adaptation for free. 
Then, the estimator of $m$ is defined as:
$$
\widehat m_{L_{opt,SS}}(x) := \Atan(\widehat p_{1, L_{opt, SS} }(x), \widehat p_{2, L_{{opt, SS}} }(x)),
$$ 
with $L_{{opt, SS}}$ defined in (\ref{Lopt-SS}).
\medskip

The oracle inequality stated in Theorem~\ref{theorem-proba-oracle} serves as a key tool for deriving convergence rates for the estimation of $m$ in the presence of ordinary smooth noise. In contrast, obtaining convergence rates in the supersmooth  noise setting is a straightforward consequence of the proofs established for the ordinary smooth case. The next theorem  presents the convergence rates.

\begin{thm} \label{adaptive:thm:circular} 
Let $p_{1}$ belongs to $\mathcal{W}(\beta_{1}, S_{1})$ and $p_{2}$ belongs to $\mathcal{W}(\beta_{2}, S_{2})$. 
\begin{enumerate}
\item Assume that the noise $\varepsilon$ is OS (satisfying (\ref{OS-circle})) with $\nu >1$. Let $q \geq 1$ and suppose that  $\min\left\{c_{0, 1} ; c_{0, 2}\right\} \geq \frac{64 (2 +q)^2 }{\min(1, \| f_{\varepsilon}^\star\|_{\ell_1})}.
$ Then,  
	for $n$ sufficiently large, 
	\begin{equation*}
		\E \Big[ d_c( \widehat{m}_{\hat L}(x), m(x) )\Big]
		\leq \frac{C_1}{\delta^2}  \,  \max \left\{    \psi_{n}^2(\beta_{1}, \nu) ,  \psi_{n}^2(\beta_{2}, \nu)\right\} , 
	\end{equation*}
	where 
	the convergence rates are     $\psi_{n}(\beta_{i}, \nu) = \big(\log n/n \big)^{\frac{\beta_{j}}{2\beta_{i} + 2\nu +1 }}$, the constant $\delta$ is defined in~\eqref{deltawx} and $C_1$ is a constant depending on $\beta_{1}, \beta_{2}, S_{1}, S_{2},c_{0,1}, c_{0,2}$ and~$f_{\varepsilon}^\star$.
	
\item Assume that the noise $\varepsilon$ is SS  (satisfying (\ref{SS-circle})). For $n$ sufficiently large, 
	\begin{equation*}		\E \Big[ d_c( \widehat m_{L_{opt,SS}}(x), m(x) )\Big]
		\leq \frac{ C_2}{\delta^2}  \,  \max \left\{  \check  \psi_{n}^2(\beta_{1}, a) ,  \check \psi_{n}^2(\beta_{2}, a)\right\} , 
	\end{equation*}
	with $ \check \psi_{n}(\beta_{i}, a) =  \left( {\log n}\right )^{\frac {- \beta_i }{a}} $,  the constant $\delta$ defined in~\eqref{deltawx} and $C_2$ is a constant depending on $\beta_{1}, \beta_{2}, S_{1}, S_{2}$ and~$f_{\varepsilon}^\star$.
	\end{enumerate}
\end{thm}

The proof of Theorem \ref{adaptive:thm:circular} is given in Section \ref{preuve-adaptive:thm:circular}.
\begin{remark}
Note that when $\beta_1 = \beta_2 = \beta$, the resulting rates are $\left(\log n / n\right)^{\frac{2\beta}{2\beta + 2\nu + 1}}$ and $\left(\log n\right)^{-2\beta / a}$, which correspond to the optimal rates for adaptive estimation of a univariate regression function and for pointwise risk in an errors-in-variables framework (see Section 3.3.2 in \cite{Meister}). The presence of a logarithmic factor in the rate under ordinary smooth noise is expected, since we deal with pointwise adaptive estimation.
\end{remark}

\section{Linear predictor} \label{section-linear}
In this section, we consider the case where the predictor $X \in [0,1]$.
 \subsection{Model}
 We study the following  model:

$$
\Theta= m(X) + \zeta \quad  \textrm{(mod $2\pi$)}
$$
with $\Theta, \zeta \in \mathbb{S}^1$ and $X \in [0,1]$. The variable $X$ is unobserved and we only have access to 
$$
Z = X + \varepsilon.
$$
 We assume that  $X$ and $\varepsilon$ are independent and  the  characteristic function of density the $f_\varepsilon$  is known. 
 From an i.i.d dataset $(Z_1, \Theta_1), \dots, (Z_n,\Theta_n)  $, we want to estimate $m: [0,1] \mapsto \mathbb{S}^1$.

\subsubsection{Assumptions}

We present our results in terms of  H\"older classes that are adapted to local estimation.
\begin{definition}\label{definition:holder.class} 
	{Let $\beta > 0$ and $\Lambda > 0$. The H\"older class $\mathcal{H}(\beta,\Lambda)$ is the set of functions $f : (0,1) \longmapsto \mathbb{R}$, such that $f$ admits derivatives up to the order $\lfloor \beta \rfloor$, and for any $(y,\widetilde{y})\in(0,1)^2$,
		\begin{align*}
			\left| \dfrac{ d^{\lfloor \beta \rfloor} f }{ (d y )^{\lfloor \beta \rfloor} } (\widetilde{y}) - \dfrac{ d^{\lfloor \beta \rfloor} f }{ (d y)^{\lfloor \beta \rfloor} } (y) \right|  \leq  \Lambda  \hspace{0.1cm}   \big|\widetilde{y}-y\big|^{\beta - \lfloor \beta \rfloor}  .
		\end{align*}
} \end{definition}
\noindent

We make the following assumptions on the covariates' noise.

\begin{assumption}  
The noise $\varepsilon$ is called ordinary smooth (OS) if there exist $r, c_{\varepsilon}, C_{\varepsilon} >0$ such that  
\begin{align}
c_{\varepsilon} \big( 1 + |t| \big)^{-r}  \leq  \big| f_{\varepsilon}^{\star}(t) \big|  \leq   C_{\varepsilon} \big( 1 + |t| \big)^{-r} , \hspace{0.2cm}  
\quad  t \in \R \,  . 
	\label{assumption:varepsilon:ordinary.smooth:form}
\end{align}
The noise $\varepsilon$ is called supersmooth (SS) if
there exist constants    ${c}_{\varepsilon} > 0$,  ${C}_{\varepsilon} > 0$, $\gamma > 0$, $\rho > 0$ and $\rho_{0}, \rho_{1} \in \R$ such that 	
\begin{align}
		{c}_{\varepsilon} \big( 1 + |t|^{2} \big)^{-\rho_{0}/2}  \exp \big( - \gamma  |t|^{\rho} \big)   \leq  \big| f_{\varepsilon}^{\star}(t) \big|  \leq    {C}_{\varepsilon} \big( 1 + |t|^{2} \big)^{-\rho_{1}/2}  \exp \big( - \gamma  |t|^{\rho} \big) , \quad   t \in \R . 
	\label{assumption:varepsilon:super.smooth:form}
\end{align}
\end{assumption}    

Due to the linear covariates, we shall use kernel estimation.

\begin{definition} \em{
		Let $K: \mathbb{R} \rightarrow \mathbb{R}$ be an integrable function on $\R$. 
		We say that $K$ is a kernel if it satisfies $\int_{\mathbb{R}} K(y)dy = 1$. For $h > 0$, we define $K_{h}(\cdot) := \frac{1}{h} K(\frac{\cdot}{h})$. 
		Also, define 
		\begin{align}
			M(K) := \left\{  
			\begin{array}{cl}
				\left\| K \right\|_{1}  &,   \textrm{ if } \left\| K \right\|_{1}  <  \infty 
				\\[0.2cm]  
				\left\| K^{\star}   \right\|_{\infty}  &,   \textrm{ otherwise }  
			\end{array}
			\right.    .  
			\label{equa:definition.M(K):form}
		\end{align}
}\end{definition}

We now consider the following assumptions on the kernel $K$, following \cite{Comte-Lacour}. A classical one to control the bias is:
\begin{assumption}\label{assumption:kernel.K} \emph{
		The kernel $K$ is of order $\mathcal{\kappa} \in \mathbb{R}_{+}$, i.e.
		\begin{enumerate}
			\item[(i)] $C_{K, \mathcal{\kappa}} :=  \int_{\mathbb{R}}(1+ |y|)^{\mathcal{\kappa}} \,  |K(y)| dy < \infty $ ;
			\item[(ii)]  $\forall k \in \left\{ 1,...,\lfloor  \mathcal{\kappa} \rfloor  \right\}$, $\int_{\mathbb{R}} y^{k} \,  K(y) dy = 0 $.
		\end{enumerate}
} \end{assumption}

Another one to ensure the estimators remain finite is:
\begin{assumption} 
\label {assumption:K^star:ordinary.smooth.noise:form}
\em{  
 If $f_{\varepsilon}^{\star}$ satisfies Assumption~\eqref{assumption:varepsilon:ordinary.smooth:form}, we assume there exists $C_{(K^{\star})} > 0$ (for simplicity, we omit the dependence on $r$ in the notation of $C_{(K^{\star})}$) such  that  
\begin{align*}  
	\int_{\R} \big|  K^{\star}(t)  \big| (1 + |t|)^{r}  dt <   \sqrt{ C_{(K^{\star})} }   \hspace{0.5cm}  \textrm{ and } \hspace{0.5cm}  \int_{\R} \big|  K^{\star}(t)  \big|^{2}  (1 + |t|)^{2r}  dt  <  C_{(K^{\star})}   .   
\end{align*}  
On the other hand, if  $f_{\varepsilon}^{\star}$ satisfies Assumption~\eqref{assumption:varepsilon:super.smooth:form}  
we will consider the sinc kernel $K(y) = \frac{\sin(y)}{\pi y}$ for $y \in \R \char92 \{0\}$ and $K(0) = \frac{1}{\pi}$,  whose Fourier Transform has a compact support, more precisely, $K^{\star}(t) = \mathbf{1}_{[-1, 1]}(t)$  for $t \in \mathbb{R}$. The sinc kernel has  $M(K) = \left\| K^{\star} \right\|_{\infty} = 1$ and satisfies Assumption~\ref{assumption:K^star:ordinary.smooth.noise:form}.      

} \end{assumption}



\subsection{Preliminary estimators}  
Now, with $j \in \left\{ 1 ; 2 \right\}$, for $x \in [0, 1]$ and a bandwidth  $h_{j} > 0$, we propose an estimator for $p_{j}$ as   
\begin{align}
\widehat{p}_{1 , h_{1}}(x) &:=  \dfrac{1}{n}  \displaystyle{ \sum_{k = 1}^{n} }  \Big(  \sin(\Theta_k) \dfrac{1}{2\pi} \int_{\R}  e^{-\mathrm{i}tx} e^{\mathrm{i} t Z_{k}} \dfrac{ K_{h_{j}}^\star(t) }{ f^{\star}_\varepsilon(t) } dt \Big) , \\
\widehat{p}_{2 , h_{2}}(x) &:=  \dfrac{1}{n}  \displaystyle{ \sum_{k = 1}^{n} }  \Big(  \cos(\Theta_k) \dfrac{1}{2\pi} \int_{\R}  e^{-\mathrm{i} tx} e^{\mathrm{i} t Z_{k}} \dfrac{ K_{h_{j}}^\star(t) }{ f^{\star}_\varepsilon(t) } dt \Big).
 \label{formula:estimate.p1-p2.fixed-h}
\end{align}
The form of our estimators for $p_1$ and $p_{2}$ are justified by the fact that, for any $x \in [0,1]$  and $h_{1}, h_{2} > 0$, \begin{align}
\E \big( \widehat{p}_{1,h_{1}}(x) \big) &= 
K_{h_{1}} \ast p_{1}(x)   \quad   \textrm{  and  }  \quad  \E \big(  \widehat{p}_{2,h_{2}}(x) \big) = K_{h_{2}} \ast p_{2}(x).    
\label{form:compute.expetation.p_{1,h1}.p_{2,h2}}
\end{align}   
Indeed we have
for $x \in  [0,1]$  and $h_{1} > 0$,   using the independence between $X$ and $\varepsilon$ and Parseval's equality:
\begin{align*}
	\E \big( \widehat{p}_{1,h_{1}}(x) \big) &= \E \Big(\sin(\Theta_{1}) \dfrac{1}{2\pi} \int_{\R}   e^{-\mathrm{i} tx} e^{\mathrm{i} t Z_{1}} \dfrac{ K_{h_{1}}^\star(t) }{ f^\star_\varepsilon(t) } dt \Big) 
	\\
	&= \dfrac{1}{2\pi} \int_{\R}   e^{-\mathrm{i} tx}  \hspace{0.1cm}  \E \Big( \E \big[   \sin(\Theta_{1}) \, | \, X_{1} \big]  \hspace{0.1cm}       e^{\mathrm{i} t X_{1}} \Big) \E \Big( e^{\mathrm{i} t  \varepsilon_{1}}  \Big) \dfrac{ K_{h_{1}}^\star(t) }{ f^\star_\varepsilon(t) } dt
	\\
	&=  K_{h_1} \ast p_{1}(x),
\end{align*}  
and the same  computations handle for $ \widehat{p}_{2,h_{2}} $.


\subsection{Bias and variance}
Let us derive the bias and variance of our estimates. 
\begin{proposition}\label{lem:point-wise:mean-var.p1-p2}
	Let $j\in\{1,2\}$. Suppose that $p_{j}$ belongs to  $\mathcal{H}(\beta_{j}, \Lambda_{j})$, with $\Lambda_{j},\beta_{j} \in \R^{*}_{+}$. Suppose that the kernel $K$ satisfies Assumption~\ref{assumption:K^star:ordinary.smooth.noise:form} and  Assumption~\ref{assumption:kernel.K} with an index $\mathcal{\kappa} \in \mathbb{R}_{+}$ such that $\mathcal{\kappa} \geq   \beta_{j}$.  Then 	\begin{align*}
		\Big| \mathbb{E}\big(\widehat{p}_{j , h_{j} }(x)\big) - p_{j}(x) \Big|  
		\leq     C_{K,\mathcal{\kappa}} \,  \Lambda_{j} \,  h_{j}^{\beta_{j}} , 
		\,  &\textrm{ and }  \,
		\Var\big( \widehat{p}_{j , h_{j} }(x) \big)  \leq  
		\widetilde V_{0}(n,h_{j}),
	\end{align*}
	with 	the constant  $C_{K,\mathcal{\kappa}}$ defined in Assumption~\ref{assumption:kernel.K} and 
	\begin{align*}
	\widetilde V_{0}(n,h_{j}) = \frac{1}{(2\pi)^{2} n} \min  \left\{ \left\|  \dfrac{ \hspace{0.01cm}  K_{h_{j}}^{\star} \hspace{0.01cm}   }{f_{\varepsilon}^{\star}}    \right\|_{2}^{2}  \left\| f_{\varepsilon}^{\star} \right\|_{1} \, ; \,  \left\|  \dfrac{ \hspace{0.01cm}  K_{h_{j}}^{\star} \hspace{0.01cm}   }{f_{\varepsilon}^{\star}}    \right\|_{1}^{2}    \right\} \hspace{0.01cm} . 
	\end{align*}
	Furthermore, we have that
\begin{enumerate}	
\item If the noise ${\varepsilon}$ is OS (satisfying~\eqref{assumption:varepsilon:ordinary.smooth:form}) with $r  >  1$, then if $h_j \leq 1$,
	\begin{align*}
		\widetilde V_{0}(n,h_{j})  \leq  \dfrac{    C }{(2\pi)^{2}} \,   n^{-1} h_{j}^{-1 - 2\r}.
	\end{align*}
\item If the noise ${\varepsilon}$ is SS (satisfying~\eqref{assumption:varepsilon:super.smooth:form}) then   	
			\begin{align*}
	\widetilde V_{0}(n,h_{j})  \leq  
	\dfrac{    C}{ (2\pi)^{2} }   \,   n^{-1} h_{j}^{(\rho - 1 )_{+} }     h_{j}^{-2\rho_{0} - 1 + \rho}  \exp \big( 2 \gamma   \,    h_{j}^{- \rho}  \big) 
	\,  .    
	\end{align*}
	\end{enumerate}
\end{proposition}

The proof of Proposition \ref{lem:point-wise:mean-var.p1-p2} is done in Section \ref{preuve-lem:point-wise:mean-var.p1-p2}.
\medskip

From Proposition \ref{lem:point-wise:mean-var.p1-p2}, we deduce by the bias variance tradeoff the following rates of convergence to estimate the function $p_i$, which are similar to the circular covariates case: 
\begin{enumerate}
\item  If the noise ${\varepsilon}$ is OS, the optimal bandwidth is of order ${n}^{-\frac{1}{2 \beta_j + 2 r +1 }}$ which yields a convergence rate of order $n^{\frac{-2 \beta_j}{2 \beta_j + 2 r +1 }}$. Consequently, one needs a data-driven bandwidth selection since the optimal bandwidth depends on $\beta_j$ which is unknown. This will be tackled in the next section.  
\item  If the noise  ${\varepsilon}$ is SS, then the optimal bandwidth does not depend on the regularity $\beta_j$ and is of order  
\begin{equation}\label{rate-pi-linear}
h_{opt, SS} := \left (\frac{\log n }{2 \gamma} \right )^{-\frac 1 \rho} 
\end{equation}

and there is no need to select a bandwidth. The subsequent rates of convergence are of order $ (\log n )^{\frac{-2 \beta_j}{\rho}}$. 
\end{enumerate}

\subsection{Adaptive estimation 
}
\label{sec:adaptive:oracle-inequalities.rate-convergence:m}

We now focus on  the case of ordinary smooth noise. Let us describe how to select in a data-driven way the bandwidths $h_1$ and $h_2$ for the estimators $\widehat{p}_{1,h_1}(x)$ and $\widehat{p}_{2,h_2}(x)$, respectively. This selection will be done among a convenient grid $\mathcal{H}$, defined as follows    
\begin{align}
	\mathcal{H} =   \left\{  h  =  k^{-1}      
 ,  k \in \mathbb{N}^{*} , k  \leq n : 	
\hspace{0.1cm}  \left\| \dfrac{K^{\star}_{h} }{ f_{\varepsilon}^{\star} }   \right\|^2_{2}   \Bigg( \left\| \dfrac{K^{\star}_{h} }{ f_{\varepsilon}^{\star} }   \right\|_{1} \Bigg)^{-2}  \geq  \dfrac{\log n}{n}    
	\right\} .  
	\label{def:bandwidth.colletion.H_n}
\end{align}    
Again, we propose a  Goldenshluger-Lepski selection procedure to select ${h}_{1}$ and ${h}_{2}$. More precisely,  for $x \in [0,1]$, we define:
\begin{align*}  
\widetilde A(h_{j},x )  = \sup_{h_{j} '\in \mathcal{H}} \left\{ \big| \widehat{p}_{j, h_{j} , h_{j}'}(x) -  \widehat{p}_{j, h_{j}'}(x) \big| - \sqrt{ \widetilde{V}_{j}(n,h_{j}) }  \right\}_{+}, 
\end{align*}  
where $ \widetilde{V}_{j}(n,h_{j}) :=   \tilde c_{0,j}   \log (n)    \widetilde V_{0}(n,h_{j})$, $\tilde c_{0,j} > 0$ are tuning parameters and \begin{align*}
\widehat{p}_{j , h_{j} , h_{j}'}(x) := \big( K_{h_{j}'} \ast \widehat{p}_{j , h_{j}} \big)(x) 
\end{align*}
so that $\widehat{p}_{j,h_{j},h_{j}'}(x) = \widehat{p}_{j , h_{j}' , h_{j}}(x)$ for any $h_{j}, h_{j}' \in \mathcal{H}$.  
Thus, an adaptive bandwidth is selected as     
\begin{align}   
\widehat{h}_{j} := \underset{h_{j} \in \mathcal{H}}{\textrm{argmin}} \left\{  \widetilde A(h_{j}, x) + \sqrt{  \widetilde{V}(n,h_{j})  }  \right\}  .    
\label{equa:adaptive.selection.h1-h2:form}
\end{align}    

The following theorem establishes an oracle-type inequality for the estimators $\widehat{p}_{j , \widehat{h}_{j}}$. 
\begin{thm}\label{adaptive:prop:upper-bound.high-Proba:p1-p2}
	Consider the collection of bandwidths $\mathcal{H}$ defined in~\eqref{def:bandwidth.colletion.H_n}. Let $j\in\{1,2\}$, $q \geq  1$ and assume that  $ \tilde c_{0,j} \geq 16 \big(2 + q \big)^{2}  \big( 1 + M(K) \big)^{2}  \dfrac{1}{ \min \left\{ \left\| f^{\star}_{\varepsilon}  \right\|_{1}, 1 \right\} } $ with $M(K) < \infty$. Then,  with probability larger than $1 - 4 n^{-q}$,  
	\begin{align*}
		\big| { \widehat{p}_{j,\widehat{h}_{j}}( x ) } - p_{j}(x)  \big|   \leq  \inf_{h_{j} \in \mathcal{H}} \Big\{  \Big( 1 + 2  M(K) \Big) \hspace{0.1cm}  
		\left\| p_{j} - K_{h_{j}}*p_{j} \right\|_{\infty}  
		+  3 \hspace{0.1cm}   \sqrt{\widetilde{V}_{j}(n,h_{j})} \Big\} \hspace{0.2cm}  .  
	\end{align*}  
\end{thm}
\noindent
The proof of Theorem~\ref{adaptive:prop:upper-bound.high-Proba:p1-p2} is given in Section~\ref{sec:adaptive:prop:upper-bound.high-Proba:p1-p2:proof}. 

\subsubsection{Rates of convergence for $m$}

In the case where the noise  ${\varepsilon}$ is OS, the final estimator of the target regression $m$ is
\begin{equation}\label{estimator-m-linear}
\widehat{m}_{\widehat{h}}(x) =  \Atan \big( \widehat{p}_{1,\widehat{h}_{1}}(x), \widehat{p}_{2,\widehat{h}_{2}}(x) \big).
\end{equation}

\noindent In the case of a SS  noise, remind that no selection level is required, and we set the bandwidths to the value $h_{opt, SS}$ defined in (\ref{rate-pi-linear}). We then define the estimator of $m$  as
$$
\widehat m_{h_{opt, SS}}(x) := \Atan (\hat p_{1, h_{opt, SS}}(x), \hat p_{2, h_{opt, SS} }(x)).
$$ 

\medskip

\begin{thm} \label{adaptive:thm:pointwise-risk.upper-bound:m}
	Let $\beta_{1} , \beta_{2} ,   \Lambda_{1}, \Lambda_{2} > 0$. Suppose that $p_{1}$ belongs to $\mathcal{H}(\beta_{1}, \Lambda_{1})$, $p_{2}$ belongs to $\mathcal{H}(\beta_{2}, \Lambda_{2})$ and the kernel K satisfies Assumption~\ref{assumption:K^star:ordinary.smooth.noise:form} and Assumption~\ref{assumption:kernel.K} with an index $\kappa \geq \max \big\{ \beta_{1}  , \beta_{2}  \big\}$.
	
	\begin{enumerate}
\item Assume that the noise ${\varepsilon}$ is OS with $r>1$.	 Let $q \geq 1$ and suppose that  $\min\left\{\tilde c_{0, 1} ; \tilde c_{0, 2}\right\} \geq 16 \big(2 + q \big)^{2}  \big( 1 + M(K) \big)^{2}  \dfrac{1}{ \min \left\{ \left\| f^{\star}_{\varepsilon}  \right\|_{1} ; 1 \right\} }~$. Then,  
	for $n$ sufficiently large, 
	\begin{equation*}
		\E \Big[ d_c( \widehat{m}_{\widehat{h}}(x), m(x) )\Big]
		\leq \frac{\widetilde C_1}{\delta^2}  \,  \max \left\{    \psi_{n}^2(\beta_{1}, r) ;   \psi_{n}^2(\beta_{2}, r)\right\} , 
	\end{equation*}
	where 
	the convergence rates are $\psi_{n}(\beta_{j}, r) = \big( \log n/n \big)^{\frac{\beta_{j}}{2\beta_{j} + 2r+1}} $,  and the constant $\delta$ is defined in~\eqref{deltawx} and $\widetilde C_1$ is a constant depending on $\beta_{1},  \beta_{2},  \Lambda_{1},  \Lambda_{2},  c_{0,1}, c_{0,2},  K$ and~$f_{\varepsilon}$.

\item Assume that the noise ${\varepsilon}$ is SS. Then for $n$ sufficiently large we have 
\begin{equation*}
		\E \Big[ d_c( \widehat m_{h_{opt, SS}}(x) , m(x) )\Big]
		\leq \frac{\widetilde C_2}{\delta^2}  \,  \max \left\{     \check \psi_{n}^2(\beta_{1}, \rho) ;  \check  \psi_{n}^2(\beta_{2}, \rho)\right\} , 
	\end{equation*}
with $ \check \psi_{n}(\beta_{j}, \rho) = \big( \log n \big)^{ - \frac{\beta_{j}}{\rho} }$ and the constant $\delta$ is defined in~\eqref{deltawx} and $\widetilde C_2$ is a constant depending on $\beta_{1}, \beta_{2}, \Lambda_{1}, \Lambda_{2},c_{0,1}, c_{0,2}, K$ and~$f_{\varepsilon}$.
	\end{enumerate}
\end{thm}

The proof of Theorem \ref{adaptive:thm:pointwise-risk.upper-bound:m} is done in Section \ref{sec:adaptive:thm:pointwise-risk.upper-bound:m:proof}.

We obtain similar convergence rates as in the case of circular predictors. 
When $\beta_1 = \beta_2 = \beta$, the resulting convergence rates are $\left(\log n / n\right)^{\frac{2\beta}{2\beta + 2r + 1}}$ and $\left(\log n\right)^{-2\beta / \rho}$, which match the minimax optimal rates for adaptive estimation of a univariate regression function and for pointwise risk in an errors-in-variables context. 


\section{Numerical results}\label{numerical-results}
In this section, we illustrate numerically our procedure. Both simulations and a real data example are considered.

\subsection{Simulations}

\subsubsection{Linear predictor case}
We consider the following linear-circular regression model:
\begin{align}
	\textrm{ (LC)}. \quad \Theta  &=  \Atan \big( 20 X - 11, (10 X - 5 )^{2} + 2 \big)  +  \zeta  \quad  \textrm{(mod $2\pi$)}, \label{regression.circular:numerical.simu:modelM1}
	\nonumber  
	\\
	Z &= X + \varepsilon ,
\end{align}
where the covariate $X \sim U([0,1])$ and the regression circular error  $\zeta$ is distributed according to a von Mises distribution $f_{v M(0,10)}$. Remind that the classical von  Mises distribution with location parameter $\mu \in  [-\pi,\pi) $ and  concentration parameter $\kappa >0$  denoted  whose density $f_{vM(\mu, \kappa)}$ is defined for any $\theta \in [-\pi,\pi)$ by
\begin{equation}\label{vmdensity}
	f_{vM(\mu, \kappa)}(\theta)=c(\kappa)  \exp \big({\kappa}  \cos( \theta - {\mu} ) \big),
\end{equation}
with $c(\kappa)$ a normalizing constant.

\medskip  

For the covariates' noise $\varepsilon$, we focus on the centered Laplace density with scale parameter $\sigma_{\varepsilon} > 0$, which has the following expression 
$	f_{\varepsilon}(x) = \dfrac{1}{2 \sigma_{\varepsilon} }  e^{- \frac{|x|}{\sigma_{\varepsilon}}}, \textrm{ with } x \in \R.$ 
Actually, the choice of the centered Laplace density is motivated by the fact that the Fourier transform of $f_{\varepsilon}$ is given by  
$	f_{\varepsilon}^{\star}(t)  =  \dfrac{1}{ 1 +  \sigma_{\varepsilon}^{2}  \,  t^{2} }  , \textrm{ with }  t  \in  \R  , $
which gives an example of ordinary smooth noises with degree of ill-posedness $r = 2$ (in Assumption \eqref{assumption:varepsilon:ordinary.smooth:form}).  

\medskip

Figure~\ref{fig:simulation:model1.ex4:sample_test1} displays an illustration of our statistical setting. 
\begin{figure}[ht!]
	\centering
	\includegraphics[scale=0.3]{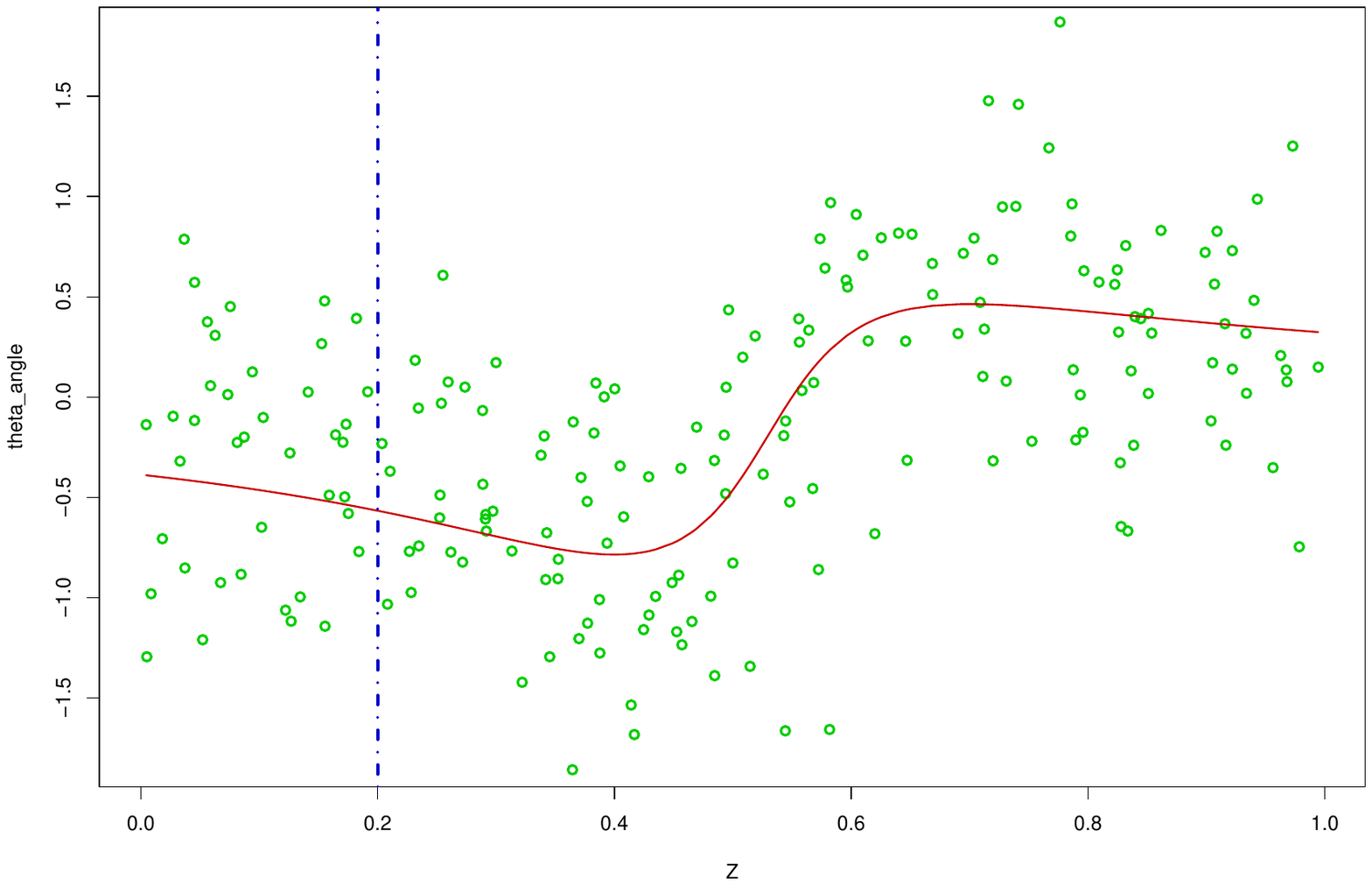}
	\vspace{-0.4cm}   
	\caption{Illustration of model (LC) with $n=200$. Simulated data $(\Theta_i )_{i=1}^{n}$ are displayed in green points. The red curve represents the regression function $m$, while the blue vertical line displays the point $x = 0.2$ where we estimate $m(x)$.}\label{fig:simulation:model1.ex4:sample_test1} 
\end{figure}

\medskip 

\noindent  

As we consider an ordinary smooth noise in the covariates, we implement the estimator $\widehat{m}_{\widehat{h}}(x) = \Atan \big( \widehat{p}_{1,\widehat{h}_{1}} (x) , \widehat{p}_{2,\widehat{h}_{2}} (x) \big)$ (defined in (\ref{estimator-m-linear})), where the bandwidths $\hat h_1$ and $\hat h_2$ are selected   according to the Goldenshluger-Lepski procedure (\ref{equa:adaptive.selection.h1-h2:form}). We consider the sinc kernel $K(x) = \frac{\sin(x)}{\pi x}$ for $x \in \R \char92 \{0\}$ and $K(0) = \frac{1}{\pi}$, for which the bandwidths are selected among the following collection  $\mathcal{H}_{n}$ defined as
$$\mathcal{H}_{n} := \left\{ k^{-1} : k \in \N , \; 1 \leq k \leq \frac{n}{\log n}  \right\}.$$ 

In the selection procedure, we need to tune the two parameters $\tilde c_{0,1}$ and $\tilde c_{0,2}$. 
To do this, we implement preliminary simulations to calibrate $\tilde c_{0,1}$ and $\tilde c_{0,2}$ by  considering model (LC), 
and for simplicity,  we consider the case $\tilde c_{0,1} = \tilde c_{0,2} = \tilde c_{0}$. We compute the following risk $\mathcal{R}$:

\begin{equation}\label{def:risk}
	\mathcal{R} := 1 - \cos \big( \widehat{m}_{\widehat{h}}(x) , m(x) \big), 
\end{equation}
and consider $\mathcal{R}$ as a function of $\tilde c_{0}$ on the following discretization grid:
\begin{equation*}
	G_{\tilde c_0} := \left\{ 0.001 ; \hspace{0.05cm} 0.0025; \hspace{0.05cm} 0.005; \hspace{0.05cm} 0.0075; \hspace{0.05cm} 0.01; \hspace{0.05cm} 0.025; \hspace{0.05cm} 0.05; \hspace{0.05cm} 0.075 ; \hspace{0.05cm} 0.1 ; \hspace{0.05cm} 0.2 ; \hspace{0.05cm} 0.3 ; \hspace{0.05cm} 0.4 ; \hspace{0.05cm} 0.6 ; \hspace{0.05cm} 0.8 ; \hspace{0.05cm} 1 ; \hspace{0.05cm} 2  ; \hspace{0.05cm} 4  \right\}.
\end{equation*}
The function $\mathcal{R}(\tilde c_0)$ is plotted  in Figure~\ref{fig:simulation:regression.circular:linear.predictor.crit.find.c01=c02.XLaplace-sd-noise=0.075:x=0.2} when estimating at $x = 0.2$, a centered Laplace distributed covariate noise $\textrm{Laplace}(0, 0.075)$, a sample size $n=200$ and $50$ Monte Carlo repetitions. The figure shows a plateau phenomenon and regarding this, the choice $\tilde c_{0,1} = \tilde c_{0,2} = 0.4$ is reasonable. In the sequel, we fix $\tilde c_{0,1} = \tilde c_{0,2} = 0.4$ for subsequent numerical simulations. 

\begin{figure}[ht!]
\begin{center}
	\includegraphics[scale=0.30]{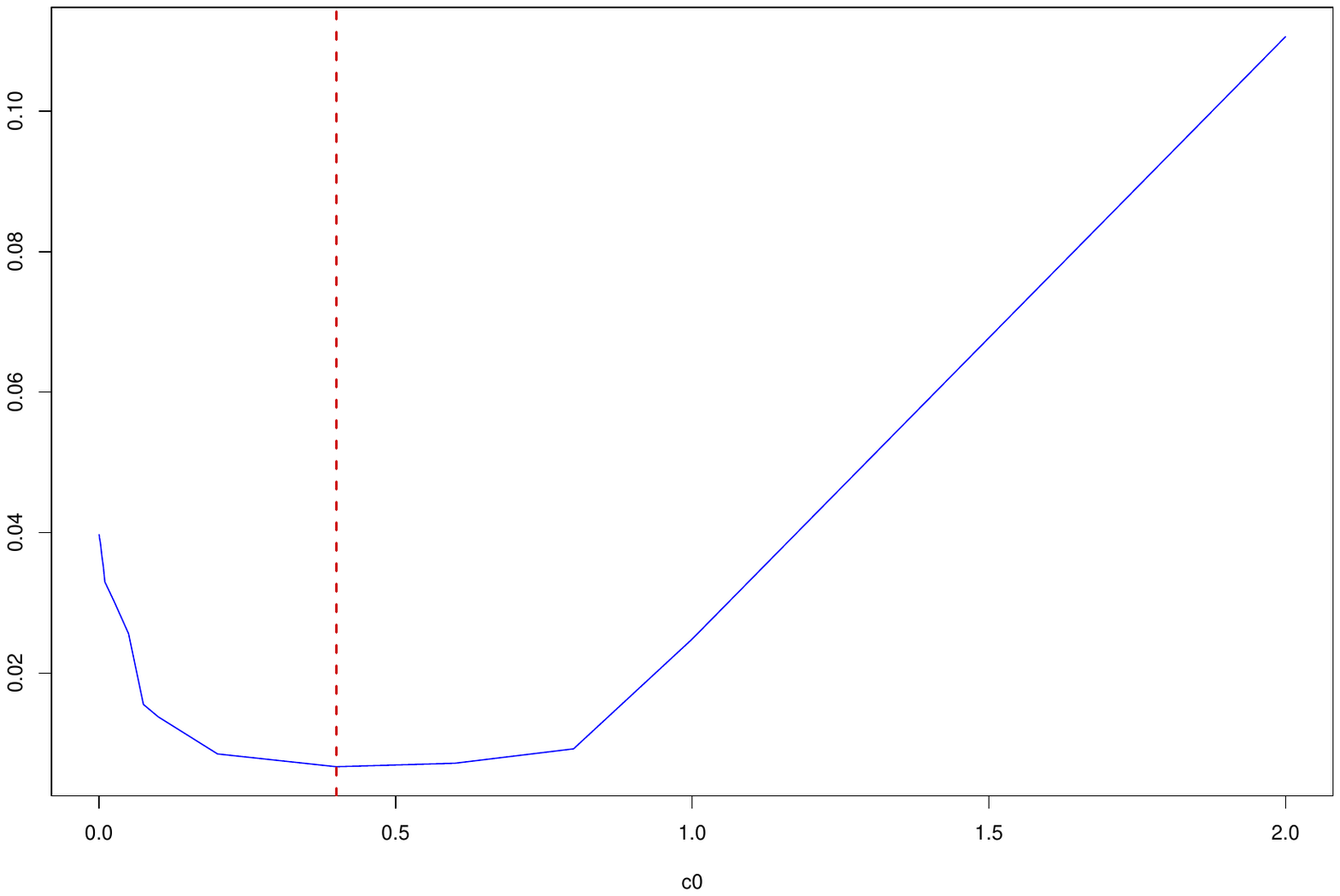}
	\caption{Plot of the function $\tilde c_{0} \in G_{\tilde c_{0}}\longmapsto\mathcal{R}(\tilde c_0)$.  The red vertical line displays the point $\tilde c_0 =0.4$.}
		\label{fig:simulation:regression.circular:linear.predictor.crit.find.c01=c02.XLaplace-sd-noise=0.075:x=0.2}
	\end{center}
\end{figure}   

We now illustrate the numerical performance of our procedure. Table~\ref{numerical.results:table:MAE.linear-predictor} presents the ME (Mean Error) of $\widehat{m}_{\widehat{h}}(x)$ over $50$ Monte Carlo runs with respect to the circular distance $d_{c}$ for model (LC) with errors-in-variables $\varepsilon \sim \textrm{Laplace}(0, 0.075)$ and  $\varepsilon \sim \textrm{Laplace}(0, 0.1)$, respectively.  In addition, when dealing with regression problems with errors in the design, it is common the consider the so-called reliability ratio (see~\cite{Fan-Truong:1993}) that is defined by 
\begin{align} \label{reliability}
	\Upsilon = \dfrac{\Var(X)}{\Var(X) +  \Var(\varepsilon) }.  
\end{align}  
This ratio $\Upsilon$ can be used to assess the amount of noise in the covariates. More specifically, when  $\Upsilon$ is close to $0$, the amount of noise in the covariates becomes  larger and this makes the deconvolution step more difficult. We calculate the reliability ratio $\Upsilon$ corresponding to our  simulations.   
 The results in Table~\ref{numerical.results:table:MAE.linear-predictor} shows that the performances of our adaptive estimator is satisfying even if the noise increases. 
\begin{table}[ht!]  
	\begin{center}
		\begin{tabular}{c  c  || c  c }
			$\sigma_{\varepsilon}$ & $\Upsilon$    & $n = 200$ & $n = 500$ 
			\\[0.2cm]
			\hline  
			& &  
			\\[-0.2cm]   
			$0.075$ & $0.88$ &  $0.0064$ & $0.0029$ 
			\\[0.2cm]
			$0.10$ &$ 0.80$ &  $0.0137$  &  $0.0048$
		\end{tabular}
	\end{center} 
	\caption{Mean errors of $\widehat{m}_{\hat h}(x)$ at point $x_{0} = 0.2$.}
	\label{numerical.results:table:MAE.linear-predictor}
\end{table} 

\subsubsection{Circular predictor cases }
For the case of circular predictors, we consider the circular-circular  regression model:
\begin{align*}
\textrm{ (CC)}. \quad	\Theta &=  \Big( \frac{1}{2} + \cos \big( x + 3 \sin (X) \big) \Big) + \zeta ,   \quad  \textrm{(mod $2\pi$)},	\\
	Z &= X + \varepsilon ,   \quad  \textrm{(mod $2\pi$)}.
\end{align*} 
We consider $X \sim$  $f_{vM(0, 0.01)}$ and a regression noise $\zeta \sim $  $f_{vM(0, 5)} $. For the noise in the covariates, 
$\varepsilon $ follows a centered wrapped Laplace with scale $\lambda_{\varepsilon} > 0$.
  The Fourier transform of the density of $\varepsilon $ is given by  
\begin{align}
	f_{\varepsilon}^{\star}(l   ) = \dfrac{ \lambda_{\varepsilon}^{2} e^{\mathrm{i} 0 \ell } }{ \big( l - \mathrm{i} \frac{\lambda_{\varepsilon}}{\kappa_{\varepsilon}} \big) 
		\big( l + \mathrm{i} \frac{\lambda_{\varepsilon}}{\kappa_{\varepsilon}} \big)
	} =  \dfrac{\lambda_{\varepsilon}^{2}}{l^{2} + \lambda_{\varepsilon}^{2}} , \quad \textrm{ for } l \in \Z,
\end{align}
giving an example of an ordinary smooth noise of degree $\nu=2$ in (\ref{OS-circle}).  Figure \ref{fig:simulation:circular.predictor:data.plot} illustrates our regression setting.

\begin{figure}[ht]
	\centering
	\includegraphics[scale=0.30]{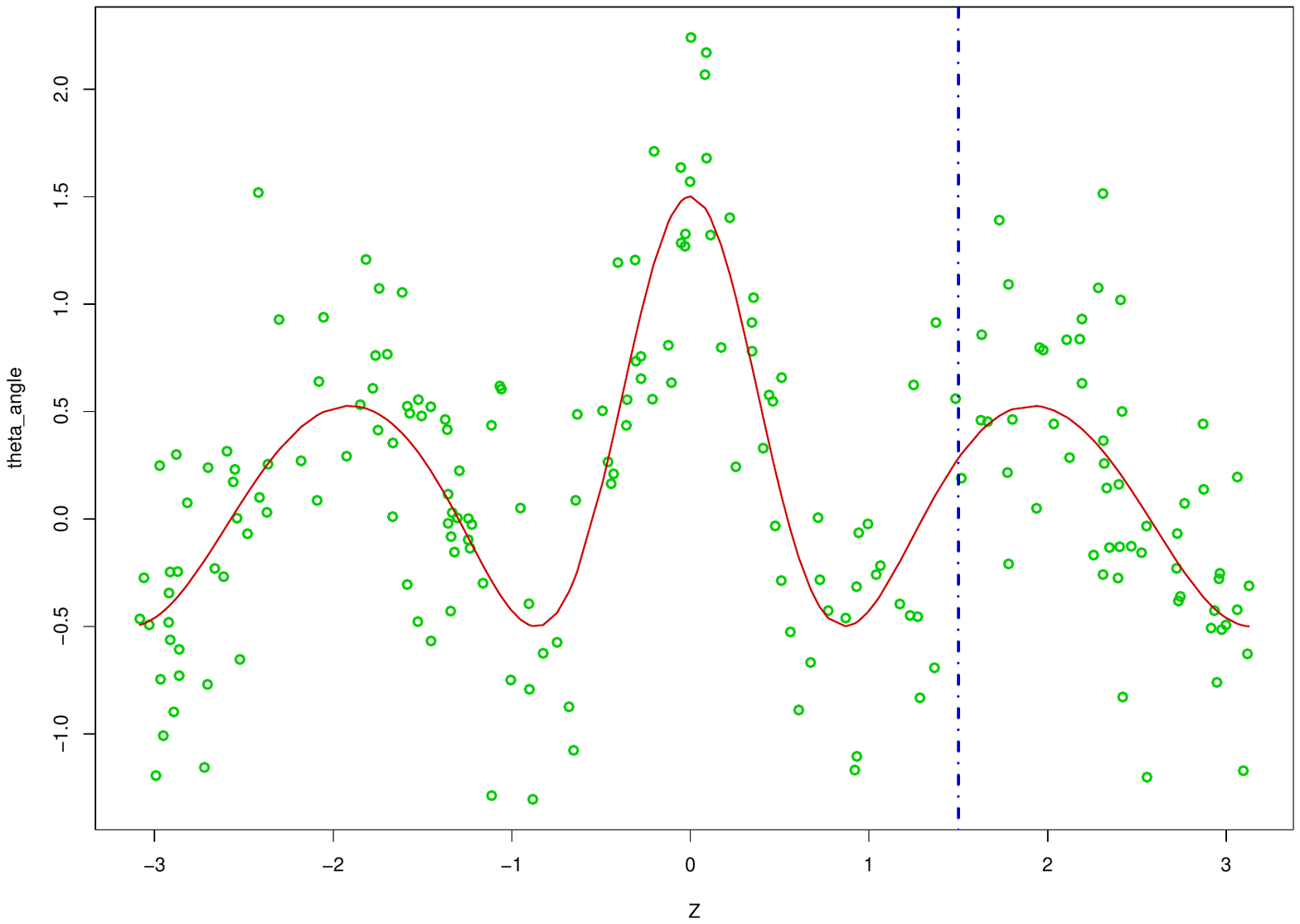}
	\vspace{-0.4cm}   
	\caption{Illustration of model (CC) with $n=200$. Simulated data $(\Theta_i )_{i=1}^{n}$ are displayed in green points. The red curve represents the regression function $m$, while the blue vertical line displays the point 
	$x = 1.5$ where we aim at estimating $m(x)$.}\label{fig:simulation:circular.predictor:data.plot}
\end{figure}

We implement the estimator $\widehat{m}_{\hat L}(x) = \Atan \big( \widehat{p}_{1,\hat {L}_{1}} (x) , \widehat{p}_{2, \hat {L}} (x) \big)$ (defined in (\ref{estimator-m-circular})), where the levels $\hat L_1$ and $\hat L_2$ are selected via the procedure (\ref{equa:adaptive.selection.L1-L2:form}). To this end, we need  to tune parameters $c_{0,1}$ and $c_{0,2}$. 
%
Similarly to the linear predictor case above, we implement preliminary simulations to calibrate them in  the case $c_{0,1} = c_{0,2} = c_{0}$ by measuring the oracle risk  $\mathcal{R}(c_0)$ given in   (\ref{def:risk}). We consider the following grid:
\[
G_{c_0} := \left\{ 0.001; \, 0.0025; \, 0.005; \, 0.0075; \, 0.01; \, 0.02; \, 0.04; \, 0.05; \, 0.06; \, 0.08; \, 0.09; \, 0.1; \, 0.2; \, 0.4; \, 0.6; \, 1; \, 2 ; \, 4 \right\}.
\]   

The function $\mathcal{R}(c_0)$ is plotted  in Figure~\ref{fig:simulation:regression.circular:circular.predictor.crit.find.c01=c02.XwrappedLaplace-sd-noise=2.75:x=1.5} when estimating at $x = 1.5$, for a covariate noise scale parameter $\lambda_{\varepsilon}=2.54$, a sample size $n=200$ and $50$ Monte Carlo runs.  A clear plateau phenomenon appears, showing that the choice of $c_{0,1} = c_{0,2} = 0.08$ is reasonable for further simulations.

\begin{figure}[ht!]
	\centering 
		\includegraphics[scale=0.30]{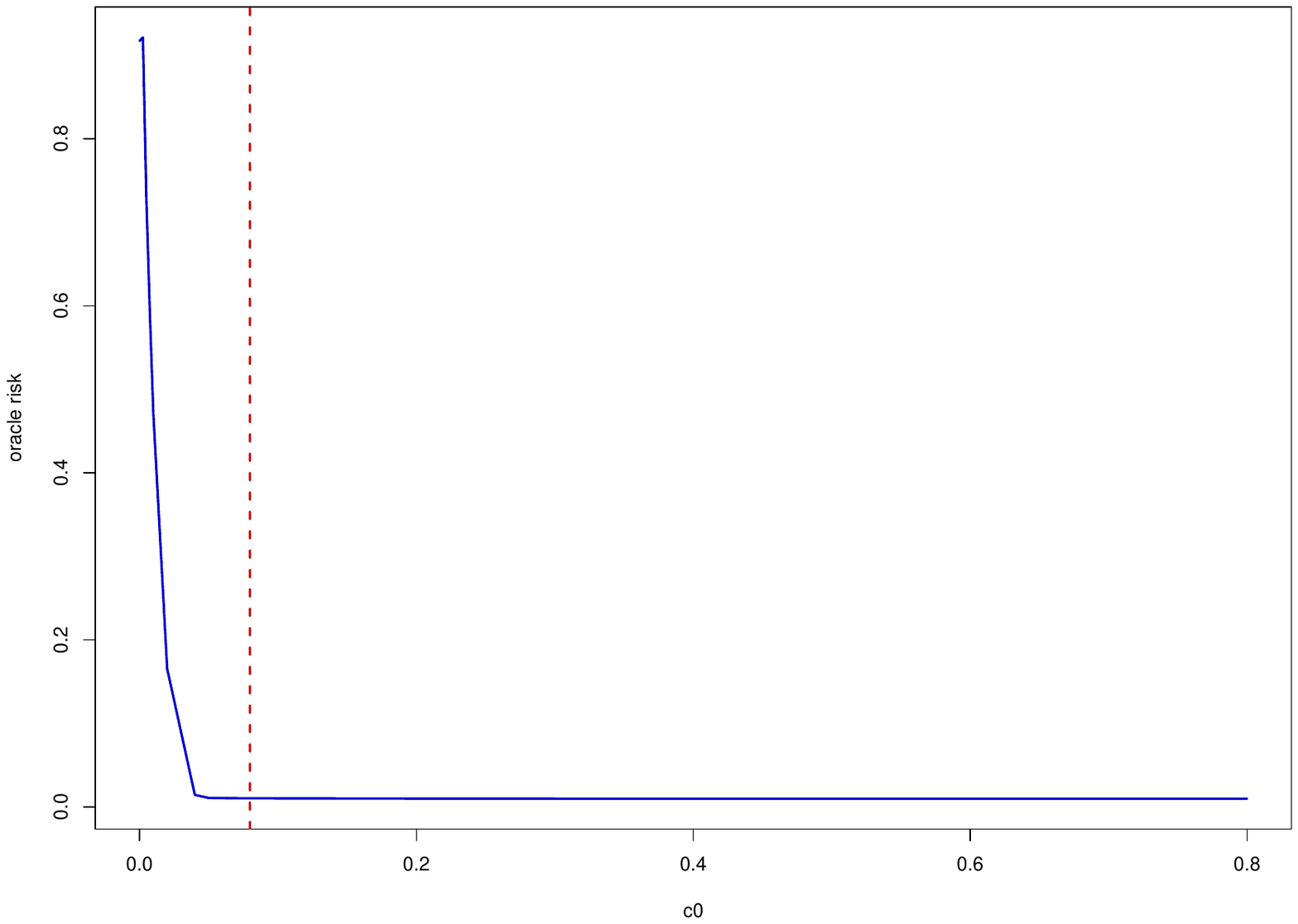}
	\caption{Plot of the function $c_{0} \in G_{c_{0}}\longmapsto\mathcal{R}(c_0)$. The red vertical line displays the point $c_0 =0.08$.}
	\label{fig:simulation:regression.circular:circular.predictor.crit.find.c01=c02.XwrappedLaplace-sd-noise=2.75:x=1.5}
\end{figure}


We also would like  to display the reliability ratio associated to each simulation. To do so, according to (\ref{reliability}) we need to find the values of $\Var(X)$ and $\Var(\varepsilon)$ which are given below for a von Mises distribution $f_{vM(0, 0.01)}$ and a centered wrapped Laplace distribution of scale parameter $\lambda_\varepsilon$, respectively: 
  $$\Var(X) = 1 - \dfrac{I_{1}(0.01)}{I_{0}(0.01)},$$ where $I_{\alpha}(\kappa)$ is the Bessel function of order $\alpha $ and 
\begin{equation}\label{var.wrappedLaplace}
	\Var(\varepsilon) = 1 - \dfrac{ \lambda_{\varepsilon}^{2} }{ \sqrt{ \big( \frac{1}{ \kappa_{\varepsilon}^{2} } + \lambda_{\varepsilon}^{2}  \big)  \big(  \kappa_{\varepsilon}^{2} + \lambda_{\varepsilon}^{2}  \big) } } = \dfrac{1}{1 + \lambda_{\varepsilon}^{2}} . 
\end{equation}
Finally, Table \ref{numerical.results:table:MAE.circular-predictor} displays our numerical performances showing again that they are again satisfying. 
\begin{table}[ht!]  
	\begin{center}
		\begin{tabular}{c  c || c  c }
			& 
			\\[0.1cm]   
			$\lambda_{\varepsilon}$ & $\Upsilon$ & $n = 200$ & $n = 500$ 
			\\[0.2cm]
			\hline  
			& &  
			\\[-0.2cm] $2.54$ & $0.88 $&  $0.0124$ & $0.0091$ 
			\\[0.2cm]
			$1.74$ & $0.80$ &  $0.0139$  &  $0.0107$
		\end{tabular}
	\end{center} 
	\caption{Mean errors  of $\widehat{m}_{\hat L}(x)$ at point $x_0=1.5$}
	\label{numerical.results:table:MAE.circular-predictor}
\end{table}

\subsection{Real data application}

We analyze the real data set described in Example~$6.3$ of Fisher's monograph \textit{Statistical Analysis of Circular Data} \cite{Fisher:1993}, which arises from experiments on the distances $(X)$ and directions $(\Theta)$ moved by $31$ small blue periwinkles (\textit{Nodilittorina unifasciata}) after transplantation downshore from the height at which they normally live. In fact, $15$ observations correspond to individuals measured one day after transplantation, and the remaining $16$ to those measured four days later. Note that no significant differences were observed between the behaviors of these two groups. Our aim is to predict direction from distance traveled, placing the problem in the framework of a regression model with a circular response and a linear predictor.   

Di Marzio et al. \cite{Marzio-Panzera-Taylor} analyzed this data set using a nonparametric kernel-weighted regression estimator, while numerous parametric estimators have also been proposed: Fisher and Lee \cite{Fisher-Lee:1992} considered the regression function, $\widehat{m}_{\textrm{FL}}(x) = 1.693 + 2 \tan^{-1} \big( -0.0066 ( x - 47.65) \big)$, while Presnell et al. \cite{Presnell} proposed $\widehat{m}_{\textrm{SPML}}(x) = \Atan \big( 0.157 + 0.049 x,  - 1.228 + 0.03 x \big) $.  More recently, Jammalamadaka and SenGupta \cite{book:Jammalamadaka-SenGupta} studied the alternative parametric model $
\widehat{m}_{\textrm{trig}}(x) = \Atan(1 + 0.021x, - 1.49 + 0.029x)$.  
To the best of our knowledge, this data set has not yet been analyzed within an errors-in-variables framework.


For this data set, we perform numerical experiments under two distinct scenarios in order to evaluate the performance of our data-driven statistical procedure.
In the first scenario, we consider that the covariate is contaminated by an ordinary-smooth measurement error with Laplace distribution, $\varepsilon \sim \textrm{Laplace}(0, \sigma_{\varepsilon})$, where the scale parameter is set to $\sigma_{\varepsilon} = 0.1$. The bandwidths $h_i$ are selected using our procedure with tuning constants $c_{0,1} = c_{0,2} = 0.4$.
In the second scenario, the covariate is observed with a supersmooth Gaussian measurement error, $\varepsilon \sim \mathcal{N}(0, \sigma_{\varepsilon}^2)$, with $\sigma_{\varepsilon} = 0.1$.

\begin{figure}[ht!]
	\centering 
	\includegraphics[scale=0.30]{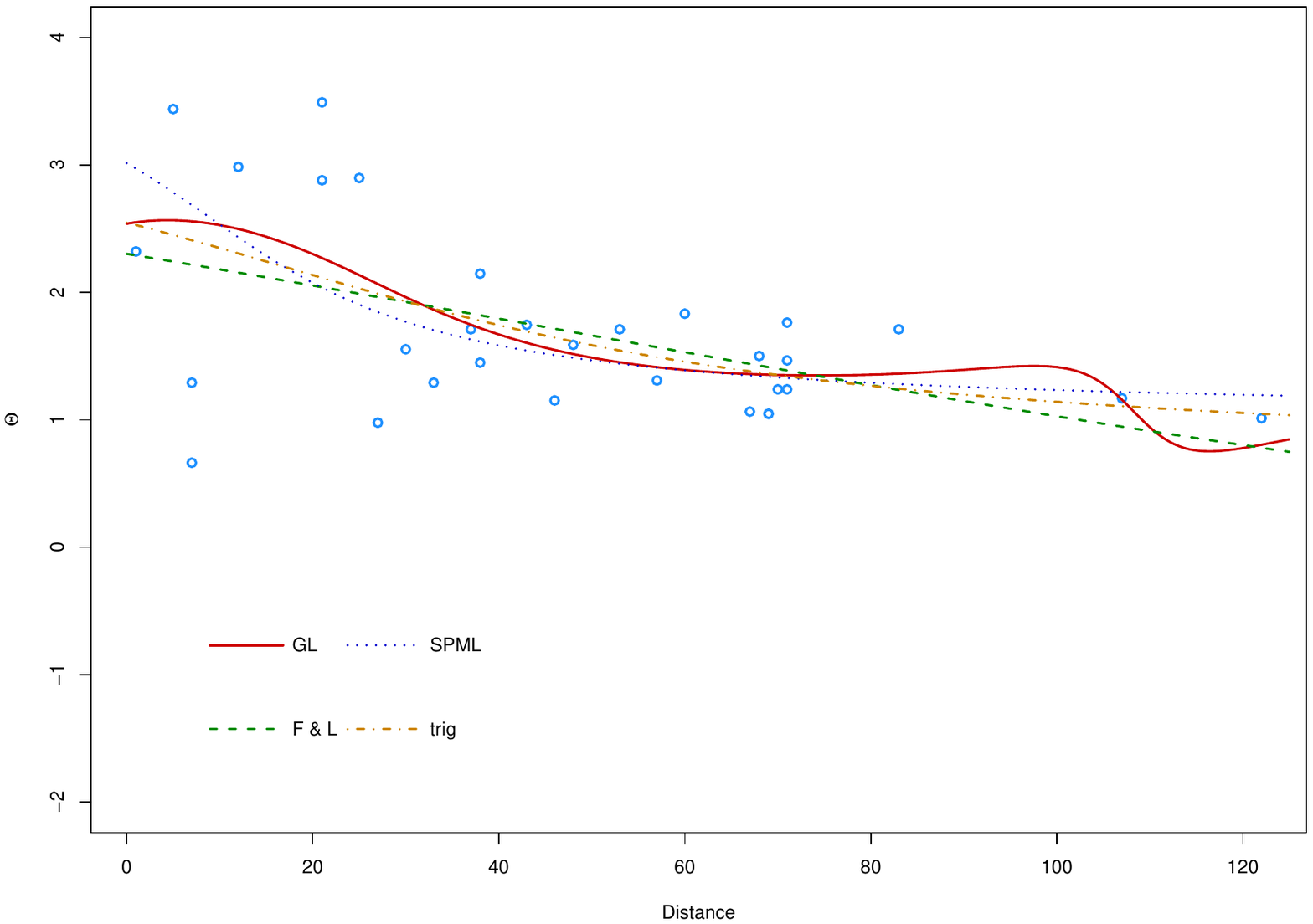}
	\caption{The red curve displays reconstruction of the regression function $m$ in Scenario $1$: $\varepsilon \sim$ Laplace$(0,\sigma_{\varepsilon})$ with $\sigma_{\varepsilon} = 0.1$.}
	\label{noise1}
\end{figure}

\begin{figure}[ht!]
	\centering 
	\includegraphics[scale=0.30]{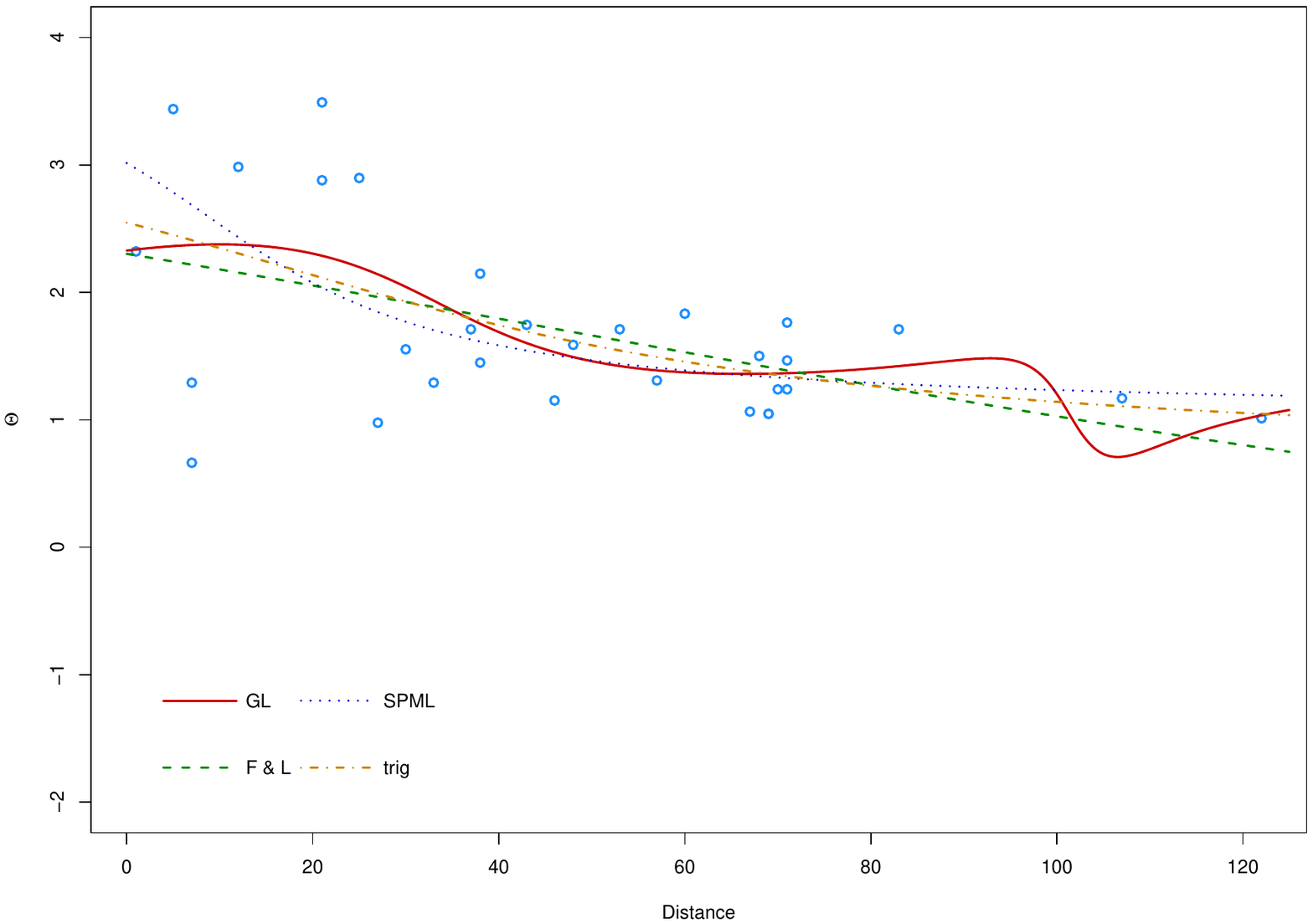}
	\caption{The red curve displays reconstruction of the regression function $m$ in Scenario $2$: $\varepsilon \sim$   $\mathcal{N}(0,\sigma_{\varepsilon}^{2})$ with $\sigma_{\varepsilon} = 0.1$.}
	\label{noise2}
\end{figure}

The estimated curves are displayed in Figures \ref{noise1} and \ref{noise2} (red). For comparison, we also report the three parametric models proposed by Fisher and Lee \cite{Fisher-Lee:1992}, Presnell et al. \cite{Presnell}, and Jammaladaka and SenGupta \cite{book:Jammalamadaka-SenGupta}; note that the Fisher-Lee model is linear. Our nonparametric estimator GL, which explicitly accounts for measurement error in the covariates, appears somewhat undersmoothed, with this effect becoming more pronounced under Gaussian noise. However, as noted by Presnell et al. \cite{Presnell}, model comparison is not straightforward in the presence of heteroscedasticity.

\section{Proofs} \label{proofs}

\subsection{Circular predictor case} \label{sec:adaptive:thm:pointwise-risk.upper-bound:m:proof}

\subsubsection{Proof of Proposition \ref{var-hat-p}} \label{preuve-proposition-CC}  
\begin{proof} 

First, let us upperbound the bias.
Assume that $p_i \in W(\beta_i, S_i)$ with $\beta_i > \frac 1 2$.  As $\E[\hat p_{i, L_i}(x)]=p_{i, L_i}(x)$ we have that  $| \E[\hat p_{i, L_i}(x) ]- p_i(x) | =|  p_{i, L_i}(x) - p_i(x)|$. And we get
\begin{eqnarray*}
2\pi | p_i(x) -p_{i, L_i}(x)| &\leq& \sum_{|l| \geq L_i} |p_i^\star(l)| =  \sum_{|l| \geq L_i} |p_i^\star(l)|(1+l^2)^{\beta_i/2}(1+l^2)^{-\beta_i/2} \\
& \leq & (\sum_{|l| \geq L_i}   |p_i^\star(l)|^2(1+l^2)^{\beta_i})^{1/2} (\sum_{|l| \geq L_i}   (1+l^2)^{-\beta_i})^{1/2} \\
& \leq &  C(S_i,\beta_i) L_i^{-\beta},
\end{eqnarray*}
where $C(S_i, \beta_i)$ is a constant depending on $S_i$ and $\beta_i$. 
\medskip

Let us prove the variance bounds. For $i \in  \{1, 2\}$ we have 
\begin{eqnarray*}
(2\pi)^2nVar(\widehat p_{i, L_i}(x)) &\leq&  \E \left  | \sum_{|l | \leq L_i}  e^{\compi l Z_1} e^{-\compi l x} \frac{1}{f^\star_{\varepsilon}(l)} \right |^2 
= \E \sum_{|l| \leq L_i}e^{\compi lZ_1} e^{-\compi lx} \frac{1}{f^\star_{\varepsilon}(l)} \sum_{|l'| \leq L_i}e^{-\compi l'Z_1} e^{\compi l'x} \frac{1}{\overline{f^\star_{\varepsilon}(l')}} \\
&= &\sum_{|l|, |l'|  \leq L_i } f_Z^\star(l-l')  e^{-\compi(l-l')x}  \frac{1}{f^\star_{\varepsilon}(l)} \frac{1}{\overline{f^\star_{\varepsilon}(l')}} \\
&\leq &  \sum_{|l|, |l'| \leq L_i} | f_Z^\star (l-l') | \frac{1}{ |f^\star_{\varepsilon}(l)|} \frac{1}{|{f^\star_{\varepsilon}(l')}|}  =
  \sum_{|\tilde l|, |l| \leq L_i} | f_Z^\star (\tilde l) |    \frac{1}{ |f^\star_{\varepsilon}(l)|} \frac{1}{|{f^\star_{\varepsilon}(l -\tilde l )|}} \\
& \leq &  \sum_{|\tilde l| \leq L_i} | f_Z^\star (\tilde l) | \left ( \sum_{|l | \leq L_i}   \frac{1}{ |f^\star_{\varepsilon}(l)|^2} \right )^{1/2} \left   ( \sum_{|l | \leq L_i}   \frac{1}{| f^\star_{\varepsilon}(l- \tilde l )|^2} \right )^{1/2} \\
& \leq &   \sum_{|\tilde l| \leq L_i} | f_Z^\star (\tilde l) |   \left ( \sum_{|l | \leq L_i}   \frac{1}{ |f^\star_{\varepsilon}(l)|^2} \right )^{1/2}   \left ( \sum_{|l | \leq 2L_i }   \frac{1}{ |f^\star_{\varepsilon}(l)|^2} \right )^{1/2} \\
& \leq & \sum_{|\tilde l| \leq L_i}  | f_{X}^\star (\tilde l) | |  f^\star_{\varepsilon}(\tilde l) |  \sum_{|l | \leq 2L_i }   \frac{1}{ |f^\star_{\varepsilon}(l)|^2} 
 \leq   \sum_{|l | \leq 2L_i }   \frac{1}{ |f^\star_{\varepsilon}(l)|^2}   \sum_{ | l| \leq L_i} | \int e^{\compi lx} f_X(x) dx| |  f^\star_{\varepsilon}( l) |  \\
& \leq &   \sum_{|l | \leq 2L_i }   \frac{1}{ |f^\star_{\varepsilon}(l)|^2}    \sum_{|l | \leq L_i }   { |f^\star_{\varepsilon}(l)|} 
 \leq  \|f^\star_{\varepsilon} \|_{\ell_1}   \sum_{|l | \leq 2L_i }   \frac{1}{ |f^\star_{\varepsilon}(l)|^2} .
\end{eqnarray*}

But we also have that 
\begin{eqnarray*}
(2\pi)^2n Var(\widehat p_{i, L_i}(x)) &\leq& \E \left | \sum_{ | l | \leq L_i} e^{il Z_1 e^{-il x}}  \frac{1}{ |f^\star_{\varepsilon}(l)| }\right |^2 
\leq   \left  ( \sum_{ | l | \leq L_i}  \frac{1}{ |f^\star_{\varepsilon}(l)| }\right )^2,
\end{eqnarray*}
hence 
$$
Var(\widehat p_{i, L_i}(x))  \leq  \frac{1}{(2\pi)^2n} \min \left ( \left  ( \sum_{ | l | \leq L_i}  \frac{1}{ |f^\star_{\varepsilon}(l)| }\right )^2,   \| f^\star_{\varepsilon} \|_{\ell_1}  \sum_{|l | \leq 2L_i }   \frac{1}{ |f^\star_{\varepsilon}(l)|^2}     \right).
$$
We now turn to the last point of the proposition which gives upper bounds for the variance when noise assumptions are made.
\medskip

 If the noise is ordinary smooth, we have provided   $\nu >1$ (to ensure that $\| f^{\star}_{\varepsilon} \|_{\ell_1}$ is finite) and $L_i\geq 1$ 
\begin{eqnarray*}
Var(\widehat p_{i, L_i}(x)) \leq \frac{1}{(2\pi)^2n} \max(1,  \|f^\star_\varepsilon\|_{\ell_1} ) \min(L_i^{2\nu+2}, L_i^{2\nu+1}) \leq  \frac{1}{(2\pi)^2n}  \max(1,  \|f^\star_\varepsilon\|_{\ell_1} )  L_i^{2\nu+1}.
\end{eqnarray*}
If the noise is supersmooth, then we get if $L_i\geq 1$
\begin{eqnarray*}
Var(\widehat p_{i, L_i}(x)) &\leq & \frac{1}{(2\pi)^2n}  \max(1,  \|f^\star_\varepsilon\|_{\ell_1} ) \min (L_i^{-2 c +1 } e^{2 b L_i^a}, L_i^{-2 c +2 } e^{2 b L_i^a}) \\
&=&   \frac{1}{(2\pi)^2n}  \max(1,  \|f^\star_\varepsilon\|_{\ell_1} ) L_i^{-2 c +1 } e^{2 b L_i^a},
\end{eqnarray*}
which completes the proof of the proposition.
\end{proof}

\subsubsection{Proof of Theorem  \ref{theorem-proba-oracle}} \label{preuve-theorem-proba-oracle}
We first start with a proposition which gives some useful concentration results for our estimators.
\begin{proposition}\label{inegalite-bernstein}
  Suppose that the noise $\varepsilon$ is OS. For all $(L_i, L_i') \in \mathcal{L}^2$, $i \in \{1, 2 \}$  and for all $ \alpha \geq 1$ we have that 
\begin{eqnarray}
\mathbb{P} [ | \widehat p_{i, L_i}(x) - p_{i, L_i}(x) | \geq c_1(\alpha) \sqrt{V(n,L_i)}] &\leq& 2 n^{-\alpha}  \label{bernstein-1} \\
\mathbb{P} [ | \widehat p_{i, L_i \wedge L_i'}(x) - p_{L_i \wedge L_i'}(x) | \geq c_1(\alpha) \sqrt{V(n,L_i')}] &\leq& 2 n^{-\alpha}  \label{bernstein-2},
\end{eqnarray}
as soon as $c_1^2(\alpha)\min({c_{0, 1}, c_{0,2}} )\geq \frac{16 \alpha^2}{\min (\| f^\star_\varepsilon \|_{\ell_1}, 1)}$.
\end{proposition}

\begin{proof}
  To prove Proposition \ref{inegalite-bernstein} we rely on the Bernstein inequality which is recalled in the next lemma. 
 \begin{lem}
 \label{lemma:Berstein.inequality}
 Let $T_1, \cdots, T_n$ be $n$ i.i.d variables and $S_n = \sum_{k=1}^n (T_k - \E(T_k))$. Then for any $\eta>0$ we have
 \begin{equation}
 \mathbb{P}(|S_n| \geq n \eta) \leq 2 \max \left \{ \exp \left (-\frac{n\eta^2}{4V} \right ), \exp \left (-\frac{n\eta}{4b} \right ) \right \},
 \end{equation}
 with $Var(T_1) \leq V$ and $|T_1| \leq b$.
\end{lem}
Now let us prove Proposition \ref{inegalite-bernstein}. We will prove it for $i=1$, as the case $i=2$ is strictly identical. 
$$
\widehat p_{1, L_1}(x)= \sum_{| l| \leq L }\frac {1}{ 2 \pi n} \sum_{j=1}^n \sin(\Theta_j) \frac{e^{\compi l Z_j}}{f^\star_{\varepsilon}(l)} e^{-\compi l x}, x \in [-\pi, \pi ).
$$
Then setting
$$
T_{L_1,j}(x):= \frac{\sin(\Theta_j)}{2 \pi} \sum_{| l| \leq L_1 } \frac{e^{\compi l Z_j}}{f^\star_{\varepsilon}(l)} e^{-\compi l x},
$$
we have
$$
\widehat p_{1, L_1}(x) = \frac{1}{n} \sum_{j=1}^n T_{L_1,j}(x).
$$
We have that
\begin{eqnarray}
|T_{L_1,j}(x)| \leq \frac{1}{2\pi}  \sum_{| l| \leq L_1 } \frac{1}{|f^\star_{\varepsilon}(l)|} =: b(L_1)\\
Var (T_{L_1,j}(x)) = n Var(\widehat p_{1, L_1}(x)) \leq n V_0(n,L_1),
\end{eqnarray}
using Proposition \ref{var-hat-p} for the last inequality. Now applying Bernstein inequality to the variables $T_{L_1,j}(x)$, we get
\begin{eqnarray*}
\mathbb{P} [| \widehat p_{1, L_1}(x)- \E(\widehat p_{1, L_1}(x))| \geq \eta(L_1)] &= &\mathbb{P}  [| \sum_{j=1}^n T_{L_1, j} (x)- \E(T_{L_1, j} (x)) | \geq n \eta(L_1)] \\
&\leq & 2 \max \left \{ \exp \left (-\frac{n\eta^2(L_1)}{4nV_0(n,L_1)} \right ), \exp \left (-\frac{n\eta(L_1)}{4b(L_1)} \right ) \right \}.
\end{eqnarray*}
Let $\alpha \geq 1$ and choose $\eta(L_1)= c_1(\alpha) \sqrt{V(n,L_1)}$. Then we get 
\begin{eqnarray*}
&\mathbb{P} & [| \widehat p_{1, L_1}(x)- \E(\widehat p_{1, L_1}(x))| \geq c_1(\alpha) \sqrt{V(n,L_1) }]  \\  &\leq& \quad 2 \max \Bigg \{ \exp \Bigg (-\frac{n c_1^2(\alpha) V(n,L_1) }{4nV_0(n,L_1)} \Bigg ),  \exp \Bigg (-\frac{n c_1(\alpha) \sqrt{V(n,L_1)} }{4b(L_1)} \Bigg ) \Bigg \}.
\end{eqnarray*}
Now let us choose $c_1(\alpha)$ in the following way. On the one hand, if $\frac{c_1^2(\alpha)c_{0,1}}{4}\geq  \alpha$ then 
\begin{eqnarray*}
\frac{ c_1^2(\alpha) V(n,L_1) }{4V_0(n,L_1)} = \frac{c_1^2(\alpha)}{4} c_{0,1} \log n  \geq \alpha \log n.
\end{eqnarray*}
On the other hand 
\begin{eqnarray*}
\frac{nc_1(\alpha) \sqrt{V(n,L_1)}}{4 b(L_1)} = \frac{c_1(\alpha) \sqrt{c_{0,1}\log n }}{4}  \frac{n \sqrt{V_0(n,L_1)}}{b(L_1)}.
\end{eqnarray*}
But 
\begin{eqnarray*}
 \frac{n \sqrt{V_0(n,L_1)}}{b(L_1)} = \sqrt{n} \min \left \{1,  \frac{  \| f^\star_{\varepsilon} \|_{\ell_1}^{1/2}} { \sum_{ | l | \leq L_1}  \frac{1}{ |f^\star_{\varepsilon}(l)| } } \left( \sum_{|l | \leq 2L_1}   \frac{1}{ |f^\star_{\varepsilon}(l)|^2}    \right)^{1/2} \right \}.
\end{eqnarray*}
Now since $L_1 \in \mathcal{L}$, this entails that 
\begin{eqnarray}
\sqrt{n} \min \left \{1,  \frac{  \| f^\star_{\varepsilon} \|_{\ell_1}^{1/2}} { \sum_{ | l | \leq L_1}  \frac{1}{ |f^\star_{\varepsilon}(l)| } } \left( \sum_{|l | \leq 2L_1 }   \frac{1}{ |f^\star_{\varepsilon}(l)|^2}    \right)^{1/2} \right\} \geq \sqrt{c_2} \sqrt{\log n },
\end{eqnarray}
with $\sqrt{c_2}= \min\{1, \| f^\star_{\varepsilon} \|_{\ell_1}^{1/2}\}$, and we get that if $ \frac{\sqrt{c_1^2(\alpha) c_{0,1} c_2}}{4} \geq \alpha$
$$
\frac{nc_1(\alpha) \sqrt{V(n,L_1)}}{4 b(L_1)}\geq \alpha \log n. 
$$
Note that the condition $ \frac{\sqrt{c_1^2(\alpha) c_{0,1} c_2}}{4}  \geq \alpha$ implies the condition $\frac{c_1^2(\alpha)c_{0,1}}{4}\geq  \alpha$, which completes the proof of the first inequality. Now for the second assertion of the proposition, we have 
\begin{eqnarray}
\widehat p_{1, L_1 \wedge L_1'}(x)= \frac 1 n \sum_{j=1}^n T_{L _1\wedge L_1', j} (x). 
\end{eqnarray}
Hence, we have 
\begin{eqnarray*}
Var(\widehat p_{L_1 \wedge L_1'}(x)) \leq n V_0(n, L_1 \wedge L_1') \leq n V_0(n, L_1') \\
 |T_{L_1 \wedge L'_1, j} (x) | \leq  \frac{1}{2\pi}  \sum_{| l| \leq L_1\wedge L_1' } \frac{1}{|f^\star_{\varepsilon}(l)|}  \leq   \sum_{| l| \leq L_1' } \frac{1}{|f^\star_{\varepsilon}(l)|} =  b(L_1'),
\end{eqnarray*}
which entails the result following exactly the same line of reasoning as for the first assertion. This concludes the proof of Proposition \ref{inegalite-bernstein}.
\end{proof}

Let us now prove Theorem \ref{theorem-proba-oracle}.

\begin{proof}
Let $L \in \mathcal L_i$ be fixed and $i \in \{1, 2\}$.
\begin{equation}
| \widehat p_{i, \hat L_i}(x)- p_i(x) | \leq \underbrace{|   \widehat p_{i, \hat L_i}(x) -    \widehat p_{i, L_i \wedge \hat L_i}(x) | }_{A_1}+ \underbrace{ |  \widehat p_{ i, L_i \wedge \hat L_i}(x) - \widehat p_{i, L_i} (x) |}_{A_2}+ |\widehat p_{i, L_i}(x) -p_i(x)|.
\end{equation}
By definition of $A(L_i,x)$ (see \ref{def-A-li}), we have 
\begin{equation}
A_1 \leq A(L_i,x) + \sqrt{V(n, \hat L_i)}.
\end{equation}
By definition of $\hat L_i$ and $A(\hat L_i,x)$, we have
$$
A_2 \leq A(\hat L_i , x) + \sqrt{V(n, L_i)}.
$$
Thus
\begin{eqnarray*}
A_1+A_2 &\leq & A(L_i,x) + \sqrt{V(n, \hat L_i)} + A(\hat L_i , x) + \sqrt{V(n, L_i)} \\
& \leq & 2A(L_i, x) + 2 \sqrt{V(n, L_i)}, 
\end{eqnarray*}
given the definition of $\hat L_i$. Thus
$$
| \widehat p_{i, \hat L_i}(x)- p_i(x) | \leq 2A(L_i, x) + 2 \sqrt{V(n, L_i)} + |\widehat p_{i, L_i}(x) -p_i(x)|.
$$
Let us study more properly the term $A(L_i,x)$. We recall that 

$$A(L_i, x)=\sup_{L_i'\in \mathcal L} \left  \{| \widehat p_{i, L_i'}(x) - \widehat p_{i, L_i \wedge L_i'}(x)| -\sqrt{V(n, L_i') } \right \}_{+}$$

Now for any $L_i' \in \mathcal{L}$, we write
\begin{eqnarray*}
\widehat p_{i, L_i'}(x) -\widehat p_{i, L_i \wedge L_i'}(x) = (\widehat p_{i, L_i'}(x) - p_{i, L_i'}(x)) - (\widehat p_{i, L_i \wedge L_i'}(x) - p_{i, L_i \wedge L_i'}(x)) + p_{i, L_i'}(x) -p_{i, L_i \wedge L_i'}(x).
\end{eqnarray*}
We have for the last term of the above decomposition

\begin{eqnarray*}
| p_{i, L_i'}(x) -p_{i, L_i \wedge L_i'}(x) | &=& \frac{1}{2\pi } \left | \sum_{|l| \leq L_i '} p_i^\star(l)e^{-\compi l x} - \sum_{|l | \leq L_i \wedge L_i'} p_i^{\star}(l) e^{-\compi lx} \right | \\
&\leq& \frac {1}{2\pi }  \sum_{|l| \geq   L _i} |p_i^\star(l)|.
\end{eqnarray*}

Hence we get for $A(L_i,x)$ that
\begin{eqnarray*}
A(L_i,x) &\leq& \sup_{L'\in \mathcal{L}} [| \widehat p_{i, L_i'}(x) - \widehat p_{i, L_i \wedge L_i'}(x) | - \sqrt{V(n,L_i')}]_{+} \\
&\leq &  \sup_{L_i'\in \mathcal{L}} [| \widehat p_{i, L_i'}(x) - p_{i, L_i'}(x) | - \sqrt{V(n,L_i')}/2]_{+} \\
&\quad +& \sup_{L_i'\in \mathcal{L}} [| \widehat p_{i, L_i \wedge L_i'}(x) - p_{i, L_i'}(x) | - \sqrt{V(n,L_i')}/2]_{+}  +  \frac {1}{2\pi }  \sum_{|l| \geq   L _i} |p_i^\star(l)|.
\end{eqnarray*}
Using Proposition \ref{inegalite-bernstein} with $c_1(\alpha)=\frac 1 2$, if $c_{0,i} \geq \frac{64 \alpha^2}{\min(1, \| f_{\varepsilon}^{\star}\|_{\ell_1})}$ we have
$$
\mathbb{P}( \sup_{L_i'\in \mathcal{L}} [| \widehat p_{i, L_i'}(x) -  p_{ L_i'}(x) | - \sqrt{V(n,L_i')}/2]_{+} >0) \leq 2 \sum_{L_i' \in \mathcal{L}} n^{-\alpha} \leq 2 n^{-\alpha+1},
$$
as well as 
$$
\mathbb{P}( \sup_{L_i'\in \mathcal{L}} [| \widehat p_{L_i \wedge L_i'}(x) - p_{i, L_i\wedge L_i'}(x) | - \sqrt{V(n,L_i')}/2]_{+} >0)  \leq 2 n^{-\alpha+1} .
$$
This entails that if one considers the set
\begin{eqnarray*}
\mathcal{E}_i&=&\left \{ \sup_{L_i'\in \mathcal{L}} [| \widehat p_{i, L_i'}(x) -  p_{i, L_i'}(x) | - \sqrt{V(n,L_i')}/2]_{+} =0 \right \}   \\
 &\bigcap&
\left \{ \forall L_i \in \mathcal{L},   \sup_{L_i'\in \mathcal{L}} [| \widehat p_{i, L_i \wedge L_i'}(x) - p_{i, L_i\wedge L_i'}(x) | - \sqrt{V(n,L_i')}/2]_{+} =0 \right \},
\end{eqnarray*}
one gets that $\mathbb{P}(\mathcal{E}_i) \geq 1- 4n^{2 -\alpha}$. 
Now let us choose $\alpha= 2 +q$. 
 On the set $\mathcal{E}_i$ we get 
 \begin{eqnarray*}
 | \widehat p_{i, \hat L_i}(x)-p_i(x)| &\leq& 2 A(L_i,x) + 2 \sqrt{V(n,L_i)} + |\widehat p_{i, L_i}(x)- p_i(x)| \\
 & \leq &   \frac {1}{2\pi }  \sum_{|l| \geq  L_i } |p_i^\star(l)|+ 2 \sqrt{V(n,L_i)} + |\widehat p_{i, L_i}(x) -p_{i, L_i}(x)| + | p_{i, L_i}(x)- p_i(x)| \\
 & \leq &   \frac {1}{2\pi }  \sum_{|l| \geq   L_i } |p_i^\star(l)| + 2 \sqrt{V(n,L_i)} + \sqrt{V(n,L_i)}/2  +\frac{1}{2\pi} \sum_{|l| \geq L_i} |p_i^\star(l)|\\
 & \leq & \frac{1}{ \pi}  \sum_{|l| \geq L_i} |p_i^\star(l)| + \frac{5}{2} \sqrt{V(n,L_i)}. 
 \end{eqnarray*}
 This concludes the proof of Theorem \ref{theorem-proba-oracle}.
\end{proof}

\subsubsection{Proof of Theorem   \ref{adaptive:thm:circular} } \label{preuve-adaptive:thm:circular}
We first need a lemma to bound the probability of a useful set.
\begin{lem}\label{proba-hat-p}
Let the noise be OS with $\nu>1$. Let $ i \in \{1, 2\}$ and  $p_i$ belongs to $\mathcal{W}(\beta_i, S_i)$. Under the assumptions of Theorem \ref{theorem-proba-oracle}, 
 if one denotes 
$$
 E_i :=\left  \{ |\widehat p_{i, \hat L_i} (x) -p_i(x) |^2 \leq C_i  \left (\frac{\log n}{n} \right)^{\frac{2\beta_i}{2\beta_i + 2\nu +1}} \right \}, 
$$
with $C_i$ depending on $S_i, \beta_i, c_{0, i}$ and $f^\star_\varepsilon$, then for $n$ large enough, we have
$$
\mathbb{P}( E_i) \geq 1- 4n^{-q}.
$$

\end{lem}
\begin{proof}

Let $i \in \{1, 2 \}$. Using Proposition \ref{var-hat-p} which yields that  $V(n,L_i) \precsim \frac{\log(n)}{n} L_i^{2\nu +1}$ and that $\sum_{|l| \geq L_i} |p^\star(l)|  \precsim L^{-\beta }$, one  defines $L_{i, {OS}}$, the value $L_i$ such that 
\begin{eqnarray}\label{L_opt}
L_{i, {OS}} := \arg \min_{L_i \in \N^*} \left \{ L_i^{-2\beta} + \frac{\log(n)}{n} L_i^{2\nu +1}  \right \}.
\end{eqnarray}
Let us prove that  $ L_{i, { OS}} \in \mathcal{L}$. Minimizing (\ref{L_opt}), one finds that $L_{i, { OS}}$  is proportional to $(\frac{n}{\log n})^{\frac{1}{2\beta_i + 2 \nu +1 }}$. Now let $L \in \mathbb{N}^*$. We have that
\begin{eqnarray*}
 \frac{  \left (\sum_{ | l | \leq L}  \frac{1}{ |f^\star_{\varepsilon}(l)| } \right )^2}{\sum_{|l | \leq 2L }   \frac{1}{ |f^\star_{\varepsilon}(l)|^2}} \precsim \frac{  \left (\sum_{ | l | \leq L} l^{\nu}\right )^2 }{ \sum_{|l | \leq 2L } l^{2\nu}}
  \precsim  \frac{(L^{\nu+1})^2}{L^{2\nu+1}} \precsim L. 
\end{eqnarray*}
Consequently, if $L \precsim \frac{n}{\log n}$ then $L \in \mathcal{L}$. But, $ L_{i, { OS}}  \varpropto (\frac{n}{\log n})^{\frac{1}{2\beta_i + 2 \nu +1 }} \leq  \frac{n}{\log n} $, hence $ L_{i, { OS}} \in \mathcal{L}$.
Consequently, using Theorem \ref{theorem-proba-oracle}, we deduce that with probability greater than $1-4n^{-q}$
\begin{equation*}
| \widehat p_{i, \hat L_i}(x) -p_i(x)|^2 = \mathcal{O} \left (\frac{\log n}{n} \right)^{\frac{2\beta_i}{2\beta_i + 2\nu +1}},
\end{equation*}
which completes the proof of Lemma \ref{proba-hat-p}.

\end{proof}

We are now ready to prove  Theorem   \ref{adaptive:thm:circular}. Using similar ideas to control $\E \left [  d_c ( \widehat{m}_{\hat L}(x), m(x))  \right ]$ proposed  in~\cite{Nguyen-PhamNgoc-Rivoirard} we will first study the following term 
$$R_{n}  :=  \E \left [  d_c ( \widehat{m}_{\hat L}(x), m(x))  \mathbf{1}_{E_1 \cap  E_2 } \right ].$$
First, we have
\begin{align*}
	R_{n}  &=  \E \Big[ \left ( 1- \cos \big(\Atan( \widehat{p}_{1, \hat L_1}(x) , \widehat{p}_{2, \hat L_2}(x))  \right. 
	\\
	& \hspace{3.5cm}  \left. -  \Atan( p_{1}(x) ,  p_{2}(x) \big)\big) \right )   \hspace{0.02cm}   \mathbf{1}_{E_1 \cap E_2}  \Big]\\
	&= 2  \hspace{0.02cm}  \E  \Big[ \sin^2  \Big( \frac{1}{2}  \left(  \Atan( \widehat{p}_{1, \hat L_1}(x) ,  \widehat{p}_{2, \hat L_2}(x))  
	\right.    
	\\
	& \hspace{3.5cm}  \left.  -  \Atan ( p_{1}(x), p_{2}(x))  \right ) \Big)  \hspace{0.1cm}   \mathbf{1}_{E_1 \cap {E}_{2 } }\Big].
\end{align*}
Note that under Assumption~\eqref{assumption:assumption.on.f_X} on the density $f_{X}$ of $X$, one has $m_{j}(x) = 0$ if and only if $p_{j}(x) = 0$. Now, we shall distinguish 3 cases.
\\[0.2cm]
\textbf{Case 1: $|m_1(x)|>0$ and  $|m_2(x)|>0$.}
\\[0.2cm] 
Let 
$\delta_1 = |m_1(x)| \quad \mbox{ and } \quad \delta_2 = |m_2(x)|,$ 
meaning that $\delta=\min( \delta_1,\delta_2)$ (see~\eqref{deltawx}). 
First, on the event $E_1 \cap E_2$, for $n$ large enough satisfying 
\begin{align*}
 \psi_{n}(\beta_{1}, \nu)  <  \delta_{1}/2
\qquad  
\mbox{and}
\qquad 
  \psi_{n}(\beta_{2}, \nu)  <  \delta_{2}/2  \hspace{0.2cm} ,  \end{align*}
we have 
$$\big| \widehat{p}_{1, \hat L_1}(x) - p_{1}(x) \big| 
\leq   C_{1}  \psi_{n}(\beta_{1}, \nu) 
<  \dfrac{\big| p_{1}(x)\big|}{2}  ,
$$ 
and 
$$
\big|   \widehat{p}_{2, \hat L_2}(x) - p_{2}(x) \big|   \leq   C_{2}     \psi_{n}(\beta_{2}, \nu)   < \dfrac{\big| p_{2}(x) \big|}{2} \hspace{0.2cm}  . 
$$ 
This implies that 
$$ \widehat{p}_{1, \hat L_1}(x)  \hspace{0.1cm}  p_{1}(x) > 0
\quad  \mbox{ and }  \quad   
 \widehat{p}_{2, \hat L_2}(x)  \hspace{0.1cm} p_{2}(x) > 0,  \hspace{0.1cm}
$$    
and finally that $ \widehat{p}_{1, \hat L_1}(x)$ and $ \widehat{p}_{2, \hat L_2}(x)$ do not vanish on $E_1 \cap E_2$.
Therefore,
\begin{align}	
	R_n= & \hspace{0.1cm} 2 \hspace{0.02cm}  \E \Big [ \sin^2 \Big ( \frac{ 1}{ 2}  \Big (   \arctan \Big( \dfrac{ \widehat{p}_{1, \hat L_1}(x)}{ \widehat{p}_{2, \hat L_2}(x)}  \Big)  -  \arctan \Big( \dfrac{p_{1}(x)}{p_{2}(x)} \Big) \Big ) \Big)   \nonumber  
	\\
	& \hspace{7cm} \times    \mathbf{1}_{ E_1 \cap E_2 } \Big]    
	\nonumber 
	\\[0.2cm] 
	\leq & \hspace{0.1cm} \frac{ 1}{2}  \hspace{0.1cm} \E \left [  \left |     \arctan \Big( \dfrac{ \widehat{p}_{1, \hat L_1}(x)}{ \widehat{p}_{2, \hat L_2}(x)}  \Big)  -  \arctan \Big( \dfrac{p_{1}(x)}{p_{2}(x)} \Big)  \right |^2  \hspace{0.1cm}    \mathbf{1}_{E_1 \cap E_2} \right]    
	\nonumber  
	\\
	\leq& \hspace{0.1cm} \E  \Big[ \Big| \arctan \Big( \dfrac{ \widehat{p}_{1, \hat L_1}(x)}{ \widehat{p}_{2, \hat L_2}(x)}  \Big)  -  \arctan \Big( \dfrac{p_{1}(x)}{ \widehat{p}_{2, \hat L_2}(x)} \Big) \Big|^{2}    \hspace{0.1cm}    \mathbf{1}_{E_1 \cap   {E}_{2 }} \Big]     \label{adaptive:thm:pointwise-risk.upper-bound:p:proof:control.on.E2n.OS}
	\\
	&+   \E \Big[ \Big| \arctan \Big( \dfrac{p_{1}(x)}{ \widehat{p}_{2, \hat L_2}(x)}  \Big)  -  \arctan \Big( \dfrac{p_{1}(x)}{p_{2}(x)} \Big) \Big|^{2}     \hspace{0.1cm}      \mathbf{1}_{E_1 \cap  E_2 }\Big]  \nonumber.
\end{align}
On the event $E_2$, for $n$ sufficiently large,
\begin{align*}  
\big| \widehat{p}_{2, \hat L_2}(x) \big|  &\geq   \big| p_{2}(x) \big| - \big| p_{2}(x) - \widehat{p}_{2, \hat L_2}(x) \big|  
> \delta_{2} -  C_{2}     \hspace{0.1cm}  \psi_{n}(\beta_{2}, \nu)      
\geq   \delta_{2}/2 \hspace{0.1cm}  ,
\end{align*} 
and using the $1$-Lipschitz continuity of  $\arctan$, we get for the first term in~\eqref{adaptive:thm:pointwise-risk.upper-bound:p:proof:control.on.E2n.OS}, since on $E_1$ one has $\big| \widehat{p}_{1, \hat L_1}(x) - p_{1}(x) \big|   \leq   C_{1}    \hspace{0.1cm}  \psi_{n}(\beta_{1}, \nu) $,  
\begin{align*}
	&\E \Big[ \Big| \arctan \Big( \dfrac{ \widehat{p}_{1, \hat L_1}(x)}{ \widehat{p}_{2, \hat L_2}(x)}  \Big)  -  \arctan \Big( \dfrac{p_{1}(x)}{ \widehat{p}_{2, \hat L_2}(x)} \Big) \Big|^{2} \hspace{0.1cm}   \mathbf{1}_{ E_1 \cap E_2 }  \Big]
	\\
	&\leq  \dfrac{4}{\delta_{2}^2} \hspace{0.1cm}  \E \Big[ \Big| \widehat{p}_{1, \hat L_1}(x) - p_{1}(x) \Big|^{2}    \mathbf{1}_{E_1 \cap E_2 } \Big]
	\\
	&\leq  \dfrac{4}{\delta_{2}^2}  \hspace{0.1cm}  \E \Big[   C_{1}    \psi_{n}^2(\beta_{1}, \nu)  \hspace{0.1cm}  \mathbf{1}_{ E_1 \cap  E_2}  \Big]   
	\\
	&\leq  \dfrac{4}{\delta_{2}^2}  \hspace{0.1cm}  C_{1}  \hspace{0.1cm}    \psi_{n}^2(\beta_{1}, \nu )  \hspace{0.1cm} \mathbb{P} \big( E_1\cap  E_2 \big) 
	\leq  \dfrac{4}{\delta_{2}^2}  \hspace{0.1cm}  C_{1}  \hspace{0.1cm}     \psi_{n}^2(\beta_{1}, \nu)  .
\end{align*}
Moreover, for the second term in~\eqref{adaptive:thm:pointwise-risk.upper-bound:p:proof:control.on.E2n.OS}, since $\dfrac{p_{1}(x)}{ \widehat{p}_{2, \hat L_2}(x)} \times  \dfrac{p_{1}(x)}{p_{2}(x)} >  0$ on $E_2$, we have 
\begin{align*}
	&\E \Big[ \Big| \arctan \Big( \dfrac{p_{1}(x)}{ \widehat{p}_{2, \hat L_2}(x)}  \Big)  -  \arctan \Big( \dfrac{p_{1}(x)}{p_{2}(x)} \Big) \Big|^{2}  \hspace{0.1cm}   \mathbf{1}_{E_1 \cap  {E}_{2 }} \Big]
	\\
	&=  \E \Big[ \Big| \arctan \Big( \dfrac{ \widehat{p}_{2, \hat L_2}(x)}{p_{1}(x)}  \Big)  -  \arctan \Big( \dfrac{p_{2}(x)}{p_{1}(x)} \Big) \Big| ^{2}  \hspace{0.1cm}    \mathbf{1}_{E_1 \cap   E_2 } \Big] 	
	\\    
	&\leq  \dfrac{1}{\big| p_{1}(x) \big|^2} \hspace{0.1cm} \E \Big[ \Big|  \widehat{p}_{2, \hat L_2}(x) - p_{2}(x)  \Big| ^{2} \hspace{0.1cm}  \mathbf{1}_{ E_1 \cap  E_2} \Big]
	\\
	&\leq    \dfrac{1}{\delta_{1}^2}  \hspace{0.1cm}  \E \Big[  C_{2}     \psi_{n}^2(\beta_{2}, \nu )    \hspace{0.1cm}  \mathbf{1}_{ E_1 \cap   E_2 }  \Big]
	\\& \leq  \dfrac{1}{\delta_{1}^2}  \hspace{0.1cm}  C_{2}    \psi_{n}^2(\beta_{2}, \nu) \hspace{0.1cm}  .
\end{align*}
\noindent Therefore, on the event $E_1 \cap   E_2$, 
for $n$ sufficiently large such that 
\begin{align*}
C_{1} \psi_{n}(\beta_{1}, \nu )    \leq  \delta_{1}/2 , \qquad
\textrm{ and } \qquad   
C_{2}   \psi_{n}(\beta_{2}, \nu )    \leq  \delta_{2}/2, 
\end{align*}
we obtain
\begin{align*}
	&\E \Big[ \big| \widehat{m}_{\hat L}(x) - m(x) \big|^{2} \hspace{0.1cm}   \mathbf{1}_{E_1 \cap   E_2}  \Big]  
	\leq   \dfrac{4}{\delta^2}   C_{1}     \psi_{n}^2(\beta_{1}, \nu)
	 +      \dfrac{1}{\delta^2}     C_{2}    \psi_{n}^2(\beta_{2}, \nu).
\end{align*}
On the other hand, on the complementary $ E_1 ^{c}  \cup   E_2^{c}$, using the fact that $\big| \Atan (w_{1},w_{2}) \big| \leq \pi$ for all $ (w_{1},w_{2})$, we can simply obtain an upper-bound as follows:
\begin{align*}
	&\E \Big[ d_c( \widehat{m}_{\hat L}(x) - m(x) ) \hspace{0.1cm}  \mathbf{1}_{E_1^{c}  \cup  E_2^{c} }  \Big]
	\\
	&\leq \frac{1}{2}    \hspace{0.1cm}  \E \Big[ \Big| \Atan \big( \widehat{p}_{1, \hat L_1}(x) , \widehat{p}_{2, \hat L_2}(x)\big)  -  \Atan \big( p_{1}(x), p_{2}(x) \big)  \Big|^{2}  
	\\
	& \hspace{6.5cm} \times  \mathbf{1}_{ E_1^{c}  \cup  E_2^{c} }  \Big]
	\\
	&\leq  2 \pi^{2} \hspace{0.1cm}  \mathbb{P} \left( E_1^{c} \right)   +  2 \pi^{2} \hspace{0.1cm}  \mathbb{P} \left( E_2 ^{c} \right) \leq   4 \pi^{2} \hspace{0.1cm} 4 \hspace{0.1cm} n^{-q}, 
\end{align*}
by  Lemma~\ref{proba-hat-p}. For $q \geq 1$, we get that $n^{-q}$ is negligible in comparison with  $   \psi_{n}^2(\beta_{1}, \nu) $ and $ \psi_{n}^2(\beta_{2}, \nu)$.
\medskip

\noindent
{\textbf{Case 2: $m_1(x) = 0$ and $|m_2(x)|>0$.}
	\\[0.2cm]   
	In this case $\delta=|m_2(x)|$.\\
	- If $m_2(x)>0$, then $p_{2}(x) > 0$ and as previously, on the event $E_1 \cap  E_2 $, for $n$ large enough, $ \widehat{p}_{2, \hat L_2}(x)>0$. Then,
	\begin{align*}
		R_n & = 2 \hspace{0.1cm}  \E \Big[ \sin^2 \Big( \frac{1}{2}  \hspace{0.1cm}    \left (  \Atan \big( \widehat{p}_{1, \hat L_1}(x) , \widehat{p}_{2, \hat L_2}(x)\big)  -  \Atan \big( p_{1}(x), p_{2}(x)\big)  \right ) \Big) 
		\\
		& \hspace{8cm} \times  \mathbf{1}_{ E_1 \cap  E_2) } \Big]
		\\[0.2cm]  
		&= 2 \hspace{0.1cm}  \E \left [ \sin^2 \left ( \frac{1}{2} \hspace{0.1cm}   \left (  \Atan \big( \widehat{p}_{1, \hat L_1}(x) , \widehat{p}_{2, \hat L_2}(x)\big)  -  0  \right ) \right) \hspace{0.1cm}   \mathbf{1}_{E_1 \cap  E_2 } \right]
		\\[0.2cm]   
		&= 2  \hspace{0.1cm}  \E \Big[ \sin^2 \Bigg( \frac{1}{2}  \hspace{0.1cm}    \Big(  \arctan \Big( \dfrac{ \widehat{p}_{1, \hat L_1}(x)}{ \widehat{p}_{2, \hat L_2}(x)}  \Big)  -  \arctan \Big( \dfrac{ p_{1}(x)}{ \widehat{p}_{2, \hat L_2}(x)}  \Big) \Big) \Bigg) \\
		& \hspace{8cm} \times    \mathbf{1}_{E_1 \cap  E_2 } \Big]  
		\\  
		&\leq  \frac{1}{2}  \hspace{0.1cm}  \E  \Big[ \Big| \arctan \Big( \dfrac{ \widehat{p}_{1, \hat L_1}(x)}{ \widehat{p}_{2, \hat L_2}(x)}  \Big)  -  \arctan \Big( \dfrac{p_{1}(x)}{ \widehat{p}_{2, \hat L_2}(x)} \Big) \Big|^{2} \hspace{0.1cm}    \mathbf{1}_{ E_1 \cap  E_2 } \Big],     
	\end{align*}
	and we conclude as for the first case.
	
	\noindent  
	- If $m_2(x) < 0$, then $p_{2}(x) < 0$ and as previously, on the event $E_1  \cap   E_2$, for $n$ large enough, $ \widehat{p}_{2, \hat L_2}(x) < 0$. Then,
	\begin{align*}
		R_n&= 2 \hspace{0.1cm}  \E \Big[ \sin^2 \Big( \frac{ 1}{ 2}  \left (  \Atan \big( \widehat{p}_{1, \hat L_1}(x) , \widehat{p}_{2, \hat L_2}(x)\big)  -  \Atan \big( p_{1}(x), p_{2}(x)\big)  \right ) \Big)  
		\\
		& \hspace{8cm} \times    \mathbf{1}_{E_1 \cap E_2} \Big]
		\\[0.2cm]  
		&= 2 \, \E \left [ \sin^2 \left ( \frac{1}{2} \hspace{0.1cm}   \left (  \Atan \big( \widehat{p}_{1, \hat L_1}(x) , \widehat{p}_{2, \hat L_2}(x)\big)  +\pi  \right ) \right)  \hspace{0.1cm}    \mathbf{1}_{E_1 \cap E_2 } \right]
		\\
		&= 2 \, \E \Big[ \sin^2 \Big( \frac{1}{2}  \hspace{0.1cm}    \Big(  \arctan \Big( \dfrac{ \widehat{p}_{1, \hat L_1}(x)}{ \widehat{p}_{2, \hat L_2}(x)}  \Big) + 2 \pi \hspace{0.1cm}   \mathbf{1}_{\{ \widehat{p}_{1, \hat L_1}(x)> 0\}}  
		\\
		& \hspace{4.25cm}  -  \arctan \Big( \dfrac{ p_{1}(x)}{ \widehat{p}_{2, \hat L_2}(x)}  \Big) \Big) \Big)   \hspace{0.1cm}    \mathbf{1}_{E_1 \cap E_2 }\Big]
		\\[0.2cm]  
		&= 2 \, \E \Big[ \sin^2 \Big( \frac{1}{2} \hspace{0.1cm}   \Big(  \arctan \Big( \dfrac{ \widehat{p}_{1, \hat L_1}(x)}{ \widehat{p}_{2, \hat L_2}(x)}  \Big)  -  \arctan \Big( \dfrac{ p_{1}(x)}{ \widehat{p}_{2, \hat L_2}(x)}  \Big) \Big) \Big) 
		\\
		& \hspace{8cm}  \times    \mathbf{1}_{E_1 \cap  E_2} \Big]    
		\\[0.2cm]
		&\leq \frac{1}{2}  \hspace{0.1cm}  \E  \Big[ \Big| \arctan \Big( \dfrac{ \widehat{p}_{1, \hat L_1}(x)}{ \widehat{p}_{2, \hat L_2}(x)}  \Big)  -  \arctan \Big( \dfrac{p_{1}(x)}{ \widehat{p}_{2, \hat L_2}(x)} \Big) \Big|^{2} \hspace{0.1cm}   \mathbf{1}_{E_1\cap E_2 } \Big],  
	\end{align*}
	and we conclude as for the first case.
}

\bigskip

\noindent
\textbf{Case 3: $|m_1(x)|>0$ and $m_2(x)=0$.}
\\[0.2cm]   
In this case $\delta=|m_1(x)|$.\\
- If $m_{1}(x)>0$, then $p_{1}(x) > 0$ and as previously, on the event $E_1  \cap E_2 $, for $n$ large enough, $ \widehat{p}_{1, \hat L_1}(x) > 0$. Then,
\begin{align*}
	R_n & = 2 \hspace{0.1cm}   \E \Big[ \sin^2 \Big( \frac{1}{2} \hspace{0.1cm}  \left (  \Atan \big( \widehat{p}_{1, \hat L_1}(x) ,   \widehat{p}_{2, \hat L_2}(x)\big)  -  \Atan \big( p_{1}(x), p_{2}(x)\big)  \right ) \Big)
	\\
	& \hspace{8cm} \times     \mathbf{1}_{E_1 \cap   E_2 } \Big]
	\\[0.2cm]  
	&= 2 \hspace{0.1cm}  \E \Big[ \sin^2 \Big( \frac{1}{2} \hspace{0.1cm}   \Big(  \arctan \Big( \dfrac{ \widehat{p}_{1, \hat L_1}(x)}{ \widehat{p}_{2, \hat L_2}(x)}  \Big) + \pi  \hspace{0.1cm}  \mathbf{1}_{ \{ \widehat{p}_{2, \hat L_2}(x) < 0 \} }  -  \frac{\pi}{2}  \Big) \Big) 
	\\
	& \hspace{8cm} \times   \mathbf{1}_{E_1 \cap   E_2 } \Big]
	\\
	&= 2 \hspace{0.1cm}  \E \Big[ \sin^2 \Big( \frac{1}{2} \hspace{0.1cm}   \Big(  \arctan \Big( \dfrac{ \widehat{p}_{2, \hat L_2}(x)}{ \widehat{p}_{1, \hat L_1}(x)}  \Big)  -  \arctan \Big( \dfrac{ p_{2}(x)}{ \widehat{p}_{1, \hat L_1}(x)}  \Big) \Big) \Big) 
	\\
	& \hspace{8cm} \times   \mathbf{1}_{E_1 \cap  E_2 } \Big],
\end{align*}
where the last equality is obtained by using for $u \in \R \char92 \{0\}$, 
$$\arctan(u)+\arctan(1/u)=\pi/2\times \textrm{sign}(u)$$ and by distinguishing the cases according to the sign of $ \widehat{p}_{2, \hat L_2}(x)$.
\\
- If $m_1(x) < 0$, then $p_{1}(x) < 0$ and as previously, on the event $E_1 \cap   E_2$, for $n$ large enough, $ \widehat{p}_{1, \hat L_1}(x) < 0$. Then, similarly,
\begin{align*}
	R_n  &= 2 \hspace{0.1cm}   \E \Big[ \sin^2 \Big( \frac{1}{2}  \hspace{0.1cm}   \left (  \Atan \big( \widehat{p}_{1, \hat L_1}(x) , \widehat{p}_{2, \hat L_2}(x)\big)  -  \Atan \big( p_{1}(x), p_{2}(x)\big)  \right ) \Big) 
	\\
	& \hspace{8cm}   \times   \mathbf{1}_{E_1 \cap  E_2 } \Big]
	\\[0.2cm]  
	&= 2  \hspace{0.1cm}  \E \Big[ \sin^2 \Big( \frac{1}{2}  \hspace{0.1cm}    \Big(  \arctan \Bigg( \dfrac{ \widehat{p}_{1, \hat L_1}(x)}{ \widehat{p}_{2, \hat L_2}(x)}  \Bigg) - \pi  \hspace{0.1cm}  \mathbf{1}_{\{ \widehat{p}_{2, \hat L_2}(x) < 0\}}  + \frac{\pi}{2}  \Big) \Big)  
	\\
	& \hspace{8cm}  \times    \mathbf{1}_{E_1 \cap  E_2 } \Big]
	\\[0.2cm]  
	&= 2  \hspace{0.01cm}  \E \Big[ \sin^2 \Big( \frac{1}{2}    \hspace{0.01cm}   \Big(  \arctan \Big( \dfrac{ \widehat{p}_{2, \hat L_2}(x)}{ \widehat{p}_{1, \hat L_1}(x)}  \Big) -\arctan \Big( \dfrac{ p_{2}(x)}{ \widehat{p}_{1, \hat L_1}(x)}  \Big) \Big) \Big) 
	\\
	& \hspace{8cm}  \times   \mathbf{1}_{E_1 \cap   E_2 } \Big].
\end{align*}
We conclude by using similar arguments presented for the second case since $ \widehat{p}_{1, \hat L_1}(x)$ (respectively $p_{1}(x)$) and  $ \widehat{p}_{2, \hat L_2}(x)$ (respectively $p_{2}(x)$) play a symmetric role.
Note that under Assumption~\eqref{deltawx}
, the case $m_1(x) = m_2(x) = 0$ cannot occur. 
This concludes the proof of the ordinary smooth noise case.

As for the supersmooth noise case, the proof is identical to that of the ordinary smooth case, and is in fact a much simpler version. Indeed, in the supersmooth setting, there is no need to select a bandwidth, and hence no need to consider the events $ E_1$ and $ E_2$. One simply follows the proof of the ordinary smooth case, omitting the indicator $\mathbf{1}_{E_1 \cap   E_2 } $ wherever it appears. This completes the proof of  Theorem \ref{adaptive:thm:circular}.

\subsection{Linear predictor case}
We now present proofs for the linear predictor case. Along this subsection, we fix $x \in [0,1]$.   
\subsubsection{Proof of Proposition \ref{lem:point-wise:mean-var.p1-p2}} \label{preuve-lem:point-wise:mean-var.p1-p2}

We first state the following helpful lemma.  
\begin{lem}(see  \cite[Lemma 1]{Comte-Lacour})
\label{lem:proof.of.lem:point-wise:mean-var.p1-p2:preliminary.lemma:bound.integral}
Let $c, s \geq 0$ and $\varrho \in \R$ such that $2 \varrho > -1$ if $c = 0$ or $s = 0$. Then, for all $\omega > 0$, there exist constants $\lambda_{1}, \lambda_{2} > 0$ such that   
\begin{align}
\lambda_{1} \hspace{0.1cm} \omega^{2\varrho + 1 - s} \exp \big( c  \hspace{0.1cm} \omega^{s} \big)  \leq  \int_{-\omega}^{\omega}  (1 + |u|^{2})^{\varrho} \exp \big( c |u|^{s} \big) du   \leq   \lambda_{2}  \hspace{0.1cm}    \omega^{2\varrho + 1 - s} \exp \big( c \hspace{0.1cm} \omega^{s} \big) .
\end{align}
\end{lem} 
Let us start proving Proposition  \ref{lem:point-wise:mean-var.p1-p2}.  We will prove it for the case $j=1$ as the other case $j=2$ is strictly identical. 

The upper bound for the bias of $\widehat{p}_{1,h_{1}}(x)$ at $x \in [0 , 1]$ is classical, for instance, one can find a proof in~\cite[Lemma 5.1]{Nguyen-PhamNgoc-Rivoirard}. 
Now, for the variance term, we have 
\begin{align*}  
	\Var \big(  \widehat{p}_{1,h_{1}}(x) \big)   \leq  \dfrac{1}{(2\pi)^{2} n}     \hspace{0.1cm}    \E \Big(   \sin^{2}(\Theta_1)   \times    \Big|  \int_{\R}  e^{- \mathrm{i}tx} e^{ \mathrm{i}tZ_j} \dfrac{ \hspace{0.1cm}  K_{h_{1}}^{\star}(t) \hspace{0.1cm}   }{f_{\varepsilon}^{\star}(t)}   dt  \Big|^2  \Big)  \leq     \frac{1}{(2\pi)^{2} n}  \left\|  \dfrac{ \hspace{0.1cm}  K_{h_{1}}^{\star} \hspace{0.1cm}   }{f_{\varepsilon}^{\star}}    \right\|_{1}^{2}  .
\end{align*}  
On the other hand, we can also write  
\begin{align*}
	(2 \pi)^{2}  n \Var \big(  \widehat{p}_{1,h_{1}}(x) \big)   
	& \leq  \E \Bigg(     \Big|  \int_{\R}  e^{- \mathrm{i} t x} e^{ \mathrm{i} t Z_{1}} \dfrac{   K_{h_{1}}^{\star} (t)       }{f_{\varepsilon}^{\star} (t) }   dt  \Big|^2  \Bigg)        
	\\
	& = \E \Bigg(    \int_{\R}  e^{- \mathrm{i}tx} e^{\mathrm{i} t Z_{1}} \dfrac{   K_{h_{1}}^{\star} (t)       }{f_{\varepsilon}^{\star} (t) }   dt    \overline{  \int_{\R}  e^{- \mathrm{i}tx} e^{ \mathrm{i} t Z_{1}} \dfrac{   K_{h_{1}}^{\star} (t)       }{f_{\varepsilon}^{\star} (t) }   dt  } \, \Bigg)   
	\\
	& =  \int_{\R} \int_{\R}  e^{- \mathrm{i} (t-u) x} \int_{\R} e^{ \mathrm{i} (t-u) v} f_{Z}(v) dv \dfrac{   K_{h_{1}}^{\star} (t)    K_{h_{1}}^{\star} (-u)    }{f_{\varepsilon}^{\star} (t)   f_{\varepsilon}^{\star} (-u) }    dt du 
	\\
	& =  \int_{\R} \int_{\R}  e^{- \mathrm{i} (t-u) x}  f^{\star}_{Z}(t-u) \dfrac{   K_{h_{1}}^{\star} (t)    K_{h_{1}}^{\star} (-u)    }{f_{\varepsilon}^{\star} (t)   f_{\varepsilon}^{\star} (-u) }    dt du 
	\\
	&\leq   \int_{\R}  \left(  \int_{\R}  \left|  \dfrac{ K_{h_{1}}^{\star}(v+u)  K_{h_{1}}^{\star}(-u)   }{ f_{\varepsilon}^{\star}(v+u) f_{\varepsilon}^{\star}(-u) }  \right|  du \right)  \hspace{0.1cm}  \big|  f_{Z}^{\star}(v) \big| \hspace{0.1cm} dv
	\\
	& \leq  \int_{\R}  \left(  \int_{\R} \left|  \dfrac{ K_{h_{1}}^{\star}(v+u) }{ f_{\varepsilon}^{\star}(v+u) }  \right|^{2}  du  \right)^{\frac{1}{2}}   \left(  \int_{\R} \left|  \dfrac{ K_{h_{1}}^{\star}(-u) }{ f_{\varepsilon}^{\star}(-u) }  \right|^{2}  du  \right)^{\frac{1}{2}}   \hspace{0.1cm}  \big|  f_{Z}^{\star}(v) \big|   \hspace{0.1cm} dv  
	\\
	&\leq \left\| \dfrac{ \hspace{0.1cm}     K_{h_{1}}^{\star} \hspace{0.1cm}   }{ f_{\varepsilon}^{\star} }   \right\|_{2}^{2}   \int_{\R}   \Big|  f_{X}^{\star}(v)  \,    f_{\varepsilon}^{\star}(v) \Big| dv
		=  \left\| \dfrac{ \hspace{0.1cm}     K_{h_{1}}^{\star} \hspace{0.1cm}   }{ f_{\varepsilon}^{\star} }   \right\|_{2}^{2}  \int_{\R}   \Big( \Big| \int_{\R} e^{\mathrm{i} vy} f_{X}(y)  dy  \Big|   \Big)  \, \big|  f_{\varepsilon}^{\star}(v) \big| dv \\
	& \leq     \hspace{0.1cm} \left\| \dfrac{ \hspace{0.1cm}     K_{h_{1}}^{\star} \hspace{0.1cm}   }{ f_{\varepsilon}^{\star} }   \right\|_{2}^{2}  \left\| f_{\varepsilon}^{\star} \right\|_{1},
\end{align*}
as $\big| e^{\mathrm{i} vy}  \big| \leq 1$ for any $v, y \in \R$ and $\int_{\R} f_{X}(y)dy = 1$. This leads to the fact that, for any $x \in [0 , 1]$,  
\begin{align*}
	\Var \big(  \widehat{p}_{1,h_{1}}(x) \big)   \leq  \dfrac{1}{(2\pi)^{2} n} \min  \left\{   \left\| \dfrac{ \,   K_{h_{1}}^{\star} \,  }{ f_{\varepsilon}^{\star} }   \right\|_{2}^{2}  \left\| f_{\varepsilon}^{\star} \right\|_{1}, \left\|  \dfrac{ \,   K_{h_{1}}^{\star} \,  }{f_{\varepsilon}^{\star}}    \right\|_{1}^{2}    \right\} =  \widetilde V_0(n,h_{1}) .   
\end{align*}  
 Now, for the ordinary-smooth noise, under Assumption~\eqref{assumption:varepsilon:ordinary.smooth:form} with $r > 1$ and the constant $c_{\varepsilon} > 0$, one has 
\begin{align*}
(2\pi)^{2} \,  n \,   V_{0}(n,h_{1})  \leq  c_{\varepsilon}^{-2} \min \left\{  \left\| f_{\varepsilon}^{\star} \right\|_{1}  \int_{\R} \big| K_{h_{1}}^{\star}(t) \big|^{2} \big( 1 + |t|  \big)^{2r} dt,      
\Big( \int_{\R} \big|  K_{h_{1}}^{\star}(t) \big|  \big( 1 + |t|  \big)^{r} dt   \Big)^{2}   
\right\}  .   
\end{align*}
Moreover, recall that $K^{\star}_{h_{1}}(t) = K^{\star}(t \, h_{1})$ for any $t \in \R$,  we have
\begin{align*}
	\left\| f_{\varepsilon}^{\star} \right\|_{1}  \int_{\R} \big|  K_{h_{1}}^{\star}(t) \big|^{2} \big( 1 + |t|  \big)^{2r} dt  &=   	\left\| f_{\varepsilon}^{\star} \right\|_{1}  \int_{\R} \big|  K^{\star}(t \, h_{1}) \big|^{2} \big( 1 + |t|  \big)^{2r} dt
	\\
	&=    \left\| f_{\varepsilon}^{\star} \right\|_{1}  \,  h_{1}^{-1}  \int_{\R} \big| K^{\star}(u) \big|^{2} \Big( 1 + \Big| \frac{u}{h_{1}} \Big|  \Big)^{2r} du
	\\[0.2cm]  
	&\leq    
	\left\| f_{\varepsilon}^{\star} \right\|_{1}  \,  h_{1}^{-1-2r}  \int_{\R} \big| K^{\star}(u) \big|^{2} \big( 1 + | u |  \big)^{2r} du  
	\\
	&\leq  
	\left\| f_{\varepsilon}^{\star} \right\|_{1}  \,  h_{1}^{-1-2r} C_{(K^{\star})} ,
\end{align*}
and 
\begin{align*}
	\left( \int_{\R} \big|  K_{h_{1}}^{\star}(t) \big|  \big( 1 + |t|  \big)^{r} dt  \right)^{2}  \leq     
	\Big( h_{1}^{-1-r} \Big)^{2}  \left(  \int_{\R} \big| K^{\star}(u) \big| \big( 1 + | u |  \big)^{r} du  \right)^{2}  \leq  
	h_{1}^{-2-2r} C_{(K^{\star})} ,   
\end{align*}
where $C_{(K^{\star})}$ is defined in Assumption~\ref{assumption:K^star:ordinary.smooth.noise:form}. 
Thus, if $0  <  h_{1}  \leq  1$ we~obtain,   
\begin{align*}
	V_{0}(n,h_{1})  &\leq  \dfrac{   c_{\varepsilon}^{-2}     C_{(K^{\star})}  \min \left\{  \left\| f_{\varepsilon}^{\star}   \right\|_{1}   h_{1}^{-1-2r},    h_{1}^{-2-2r}    \right\}    }{(2\pi)^{2}}   \hspace{0.1cm}    n^{-1}     
	\\   
	&\leq  \dfrac{   c_{\varepsilon}^{-2}     C_{(K^{\star})}  \max \left\{  \left\| f_{\varepsilon}^{\star} \right\|, 1  \right\}  \times  \min \left\{  h_{1}^{-1-2r},    h_{1}^{-2-2r}    \right\}    }{(2\pi)^{2}} \hspace{0.1cm}  n^{-1}   
	\\   
	&\leq  \dfrac{   c_{\varepsilon}^{-2}     C_{(K^{\star})}  \max \left\{  \left\| f_{\varepsilon}^{\star} \right\|, 1  \right\}  }{(2\pi)^{2}} \hspace{0.1cm}  n^{-1} h_{1}^{-1 - 2\r} 
	\,   .                                                                                                                                                                       
\end{align*}
 For the super-smooth noise, under Assumption~\eqref{assumption:varepsilon:super.smooth:form} one has  
\begin{align*}
		&(2\pi)^{2}  n   V_{0}(n,h_{1}) 
		\\
		\leq  \,    &\widetilde{c}_{\varepsilon}^{-2} \min \left\{  \left\| f_{\varepsilon}^{\star} \right\|_{1}  \int_{\R} \big| K_{h_{1}}^{\star}(t) \big|^{2} \big( 1 + |t|^{ 2 } \big)^{\rho_{0}}  \exp \big(  2 \gamma |t|^{\rho}  \big)  dt,    
		\Big( \int_{\R} \big|  K_{h_{1}}^{\star}(t) \big|  \big( 1 + |t|^{ 2 } \big)^{ \frac{\rho_{0}}{2}}  \exp \big(  \gamma |t|^{\rho}  \big)  dt   \Big)^{2}   
		\right\}  .   
	\end{align*}    
Using the sinc  kernel for which we have  $K^{\star}(t) = \mathbf{1}_{[-1, 1]}(t)$  for $t \in \mathbb{R}$
, Lemma~\ref{lem:proof.of.lem:point-wise:mean-var.p1-p2:preliminary.lemma:bound.integral} with $\lambda_{2} > 0$ (and similar computations as in \cite{Comte-Lacour} page 589), one gets   
\begin{align*}
	\left\| f_{\varepsilon}^{\star} \right\|_{1}  \int_{\R} \big| K_{h_{1}}^{\star}(t) \big|^{2} \big( 1 + |t|^{ 2 } \big)^{\rho_{0}}   \exp \big(  2 \gamma |t|^{\rho}  \big)  dt
\leq  \lambda_{2}    \left\| f_{\varepsilon}^{\star} \right\|_{1} \,     h_{1}^{-2\rho_{0} - 1 +  \rho} \exp \Big( 2 \gamma h_{1}^{-\rho} \Big).
\end{align*}
On the other hand, we have 
\begin{align*}
	\left( \int_{\R} \big|   K_{h_{1}}^{\star}(t)   \big|  \big( 1 + |t|^{ 2 } \big)^{\frac{\rho_{0}}{2} }   \exp \big(  \gamma |t|^{\rho}  \big)   dt  \right)^{2}  
	\leq     \lambda_{2}^{2}  \,  h_{1}^{- 2 \rho_{0}  - 2 + 2 \rho} \exp (2 \gamma \,    h_{1}^{- \rho}).
\end{align*}
Hence, we obtain the following upper-bound for $V_{0}(n,h_{1})$ in the super-smooth noise case, 
\begin{align*}
	V_{0}(n,h_{1})  &\leq  \dfrac{  \hspace{0.1cm} (\lambda_{2} + \lambda_{2}^{2})  \,    \widetilde{c}_{\varepsilon}^{-2}  \,  \max   \left\{  \left\| f_{\varepsilon}^{\star} \right\| ; 1  
		\right\}      }{(2\pi)^{2}} \,   n^{-1}  \,  \min \left\{ 1 ; h_{1}^{-1 + \rho} \right\} h_{1}^{-2 \rho_{0} - 1 + \rho} \,  \exp \big( 2 \gamma   \,  h_{1}^{-\rho}  \big) 
		\\
		&=   \dfrac{  (\lambda_{2} + \lambda_{2}^{2}) \,   \widetilde{c}_{\varepsilon}^{-2}  \,  \max   \left\{  \left\| f_{\varepsilon}^{\star} \right\| ; 1  \right\}    }{(2\pi)^{2}} \,  n^{-1}  \,  h_{1}^{(\rho-1)_{+} }   \,    h_{1}^{- 2  \rho_{0} - 1 + \rho} \,  \exp \big( 2 \gamma   \,   h_{1}^{-\rho}  \big)  
		\,   .    
\end{align*}
Similar arguments apply for $\Var \big(  \widehat{p}_{2,h_{2}}(x) \big)$ and we conclude the proof of Proposition~\ref{lem:point-wise:mean-var.p1-p2}.  
\subsubsection{Proof of Theorem  \ref{adaptive:prop:upper-bound.high-Proba:p1-p2}  
}  
\label{sec:adaptive:prop:upper-bound.high-Proba:p1-p2:proof}   

\medskip  

First, we will establish some concentration results for  $\widehat{p}_{j,h_{j}}$ with  $j\in\{1,2\}$ in the following proposition.  
\begin{proposition}
	\label{proposition:apply.Bernstein.inequa:study.adaptation}
	Let $j\in\{1,2\}$.	Under the Assumptions of Theorem~\ref{adaptive:prop:upper-bound.high-Proba:p1-p2}, for all $(h_{j}, h_{j}') \in \mathcal{H}^2_{}$, for all $\alpha  \geq  1$,    
	\begin{align*}
		\mathbb{P} \left( \big| \widehat{p}_{j, h_{j}}(x) - \E \big[ \widehat{p}_{j, h_{j} }(x) \big]  \big|  >  \tilde c_{1}(\alpha)    \sqrt{\widetilde{V}(n,h_{j}')} \right)  &\leq  2 \hspace{0.1cm}  n^{- \alpha} \hspace{0.1cm}  ,
		\\
		\mathbb{P} \left( \big| \widehat{p}_{j, h_{j} ,h_{j}'}(x) - \E \big[ \widehat{p}_{j, h_{j} ,h_{j}'}(x) \big] \big|  >  \tilde c_{1}(\alpha)    \left\| K^{\star}   \right\|_{\infty}   \sqrt{\widetilde{V}(n,h_{j}')} \right)  &\leq  2 \hspace{0.1cm}  n^{- \alpha} \hspace{0.1cm}  ,
	\end{align*}
	as soon as $\tilde c_{1}(\alpha)$ satisfying $\tilde c^2_{1}(\alpha) \min \left\{\tilde  c_{0;1} ; \tilde  c_{0;2} \right\}   \geq  \frac{16 \alpha^{2}}{  \min  \left\{     \left\| f_{\varepsilon}^{\star} \right\|_{1},      1  \right\} }$.
\end{proposition}
 \begin{proof} Let us prove Proposition~\ref{proposition:apply.Bernstein.inequa:study.adaptation}. We only consider the case $j=1$ as the case $j=2$ is identical. 
For $1 \leq k \leq n$, we define the random variables$$U_{k,1}(x) := \sin (\Theta_k) \dfrac{1}{2\pi} \displaystyle{ \int_{\R} } e^{-\mathrm{i} tx} e^{\mathrm{i} t Z_{k}} \dfrac{ K_{h_{j}}^\star(t) }{ f^{\star}_\varepsilon(t) } dt, $$ such that, $\widehat{p}_{1 , h_{1}}(x) = \frac{1}{n} \sum_{k=1}^{n} U_{k , 1}(x)$. Notice that $\mathbb{E} \big( U_{k , 1}(x) \big) = \big( K_{h_{1}} \ast p_{1} \big)(x)$. 
We have for any $x \in  [0,1] $,   
\begin{align}
	\big| U_{k,1}(x) \big| = \left|  \sin(\Theta_{k})  \dfrac{1}{2\pi} \displaystyle{ \int_{\R} } e^{-\mathrm{i} tx} e^{\mathrm{i} t Z_{k}} \dfrac{ K_{h_{1}}^\star(t) }{ f^{\star}_\varepsilon(t) } dt  \right|  \leq \dfrac{1}{2\pi} \displaystyle{ \int_{\R} } \left|  \dfrac{ K_{h_{j}}^\star(t) }{ f^{\star}_\varepsilon(t) } \right| dt  = \dfrac{1}{2\pi} \left\|  \dfrac{ K_{h_{j}}^\star }{ f^{\star}_\varepsilon } \right\|_{1}  =:  b(h_{1}) , 
\end{align}  
and $\Var \big( U_{k , 1}(x) \big) = n \Var \big( \widehat{p}_{1,h_{1}} (x) \big) \leq n \widetilde V_{0}(n , h_{1})$. Now, applying Lemma~\ref{lemma:Berstein.inequality} to the $U_{k,1}$'s, we obtain for $\eta(h_{1}) > 0$, 
\begin{align*}
	\mathbb{P} \left(  \Big| \widehat{p}_{1 , h_{1}}(x) - \mathbb{E} \big( \widehat{p}_{1 , h_{1}}(x) \big) \Big|  \geq  \eta(h_{1})  \right)  &=  \mathbb{P} \left(  \Big| \sum_{k=1}^{n} U_{k,1}(x) - \mathbb{E} \big( U_{k,1}(x) \big) \Big|  \geq  n \eta(h_{1}) \right)  
	\\
	&\leq  2 \max  \left\{   \exp \Big( - \dfrac{n \times \big( \eta(h_{1}) \big)^{2}}{ 4 n \widetilde V_{0}(n, h_{1}) } \Big) \hspace{0.1cm}  ;    \hspace{0.1cm}   \exp \Big( - \dfrac{ n \times  \eta(h_{1}) }{ 4 b(h_{1}) } \Big)  \right\}  . 
\end{align*} 
For $\alpha \geq 1$, choose $\eta(h_{1}) = \tilde c_{1}(\alpha) \sqrt{ \widetilde{V}(n , h_{1}) }$, with 
$	\widetilde{V}(n , h_{1})  = \tilde  c_{0,1}  \log(n)  \widetilde V_{0}(n, h_{1})  . $
Then, 
\begin{align}
	& \mathbb{P}   \left(  \Big| \widehat{p}_{1 , h_{1}}(x) - \mathbb{E} \big( \widehat{p}_{1 , h_{1}}(x) \big) \Big|  \geq  \tilde c_{1}(\alpha) \sqrt{ \widetilde{V}(n , h_{1}) }    \right)  \nonumber  
	\\
	\leq & \hspace{0.2cm}  2 \max  \left\{   \exp \Bigg( - \dfrac{n  \hspace{0.1cm} \tilde c^2_{1}(\alpha)   \widetilde{V} (n , h_{1})  }{ 4 n \widetilde V_{0}(n, h_{1}) } \Bigg),  \exp \Bigg( - \dfrac{ n   \tilde c_{1}(\alpha) \sqrt{ \widetilde{V}_{1}(n , h_{1}) }  }{ 4 b(h_{1}) } \Bigg)  \right\}  . 
	\label{bernstein.inqua:byBerstein}
\end{align}
We then choose $\tilde c_{1}(\alpha)$. 
 First, $\tilde c_{1}(\alpha)$ is chosen such that 
\begin{align}
	\dfrac{n  \hspace{0.1cm} \tilde  c_{1}^2(\alpha)  \hspace{0.1cm}  \widetilde{V}(n , h_{1})  }{ 4 n \widetilde V_{0}(n, h_{1}) }  =  \dfrac{n  \hspace{0.1cm} \tilde c_{1}^2(\alpha)  \hspace{0.1cm} \tilde  c_{0,1} \log(n)\widetilde V_{0} (n , h_{1})  }{ 4 n \widetilde V_{0}(n, h_{1}) }  \geq  \alpha \log(n) , 
	\label{bernstein.inqua:maximize.1}
\end{align}
that is $\tilde c_{1}(\alpha)$ satisfies $\tilde c_{1}^2(\alpha) \hspace{0.1cm} \tilde c_{0,1} \geq 4 \alpha$.  
\\  
  Secondly, we  can write 
\begin{align*}
	\dfrac{ n   \tilde c_{1}(\alpha) \sqrt{ \widetilde{V}_{1}(n , h_{1}) }  }{ 4 b(h_{1}) }  =  \dfrac{ \tilde c_{1}(\alpha) \sqrt{\tilde c_{0,1}} }{4}  \sqrt{ \log(n) }   \dfrac{ n  \sqrt{ V_{0}(n , h_{1}) }}{ b(h_{1}) } , 
\end{align*}
and for $h_{1} \in \mathcal{H}$, 
\begin{align*}  
	n  \dfrac{  \sqrt{ V_{0}(n , h_{1}) }}{ b(h_{1}) }  
	&=  n  \dfrac{1}{ 2 \pi  \sqrt{n}  }  \hspace{0.1cm}  \min  \left\{    \left\| \dfrac{ \hspace{0.1cm}     K_{h_{1}}^{\star} \hspace{0.1cm}   }{ f_{\varepsilon}^{\star} }   \right\|_{2}  \left\| f_{\varepsilon}^{\star} \right\|_{1}^{1/2}     \hspace{0.1cm}   ;  \hspace{0.1cm}   \left\|  \dfrac{ \hspace{0.1cm}  K_{h_{1}}^{\star} \hspace{0.1cm}   }{f_{\varepsilon}^{\star}}    \right\|_{1}    \right\}  \times 2 \pi  \left( \left\| \dfrac{ K_{h_{1}}^{\star} }{ f_{\varepsilon}^{\star} } \right\|_{1}  \right)^{-1} 
	\\  
	&=  \sqrt{n}      \min  \left\{    \left\| \dfrac{ \hspace{0.1cm}     K_{h_{1}}^{\star}    }{ f_{\varepsilon}^{\star} }   \right\|_{2}  \left\| f_{\varepsilon}^{\star} \right\|_{1}^{1/2}   \left( \left\| \dfrac{ K_{h_{1}}^{\star} }{ f_{\varepsilon}^{\star} } \right\|_{1}  \right)^{-1}     \hspace{0.01cm}   ;  \hspace{0.01cm}   1  \right\}  \geq  \sqrt{c_{3}}\sqrt{\log(n)} ,  
\end{align*}   
with $\sqrt{c_{3} }=  \min  \left\{     \left\| f_{\varepsilon}^{\star} \right\|_{1}^{1/2}; 1  \right\}$, then we get 
\begin{align}
	\dfrac{ n   \tilde c_{1}(\alpha) \sqrt{ \widetilde{V}_{1}(n , h_{1}) }  }{ 4 b(h_{1}) }  &=  \dfrac{ \tilde c_{1}(\alpha) \sqrt{\tilde c_{0,1}} }{4}  \sqrt{ \log(n) }  \dfrac{ n  \sqrt{ V_{0}(n , h_{1}) }}{ b(h_{1}) }    \geq  \alpha  \log(n) 
	\label{bernstein.inqua:maximize.2}
\end{align}   
provided that $\tilde c_{1}(\alpha) \sqrt{\tilde c_{0,1}  c_{3} }   \geq  4 \nu$. Note that this condition also ensures the constraint $\tilde c_{1}^2(\alpha) \tilde c_{0,1}  \geq 4 \alpha $.
Now, combining \eqref{bernstein.inqua:maximize.1} and \eqref{bernstein.inqua:maximize.2}, we get from \eqref{bernstein.inqua:byBerstein} for any $\alpha \geq 1$ 
\begin{align}
	\mathbb{P} \left( \Big| \widehat{p}_{1,h_{1}}(x) - \mathbb{E} \big( \widehat{p}_{1,h_{1}}(x) \big) \Big| > \tilde c_{1}(\alpha)  \sqrt{\widetilde{V}(n,h_{1})} \right)  \nonumber 
	&\leq 2 n^{-\alpha}.
	\label{proof:prop:apply.Bernstein.inequa:P(Omega_{1,n}^{c})}
\end{align}
	Now, for the second assertion of Proposition~\ref{proposition:apply.Bernstein.inequa:study.adaptation}, for $1 \leq k \leq n$, 
	we define random variables  
	\[  
	\widetilde{U}_{k , 1}(x) := \sin(\Theta_k) \dfrac{1}{2\pi} \displaystyle{ \int_{\R} } e^{-\mathrm{i} tx} e^{\mathrm{i} t Z_{k}} \dfrac{  K_{h'_{1}}^\star(t)   K_{h_{1}}^\star(t) }{ f^{\star}_\varepsilon(t) } dt, \]
	hence, $\widehat{p}_{1,h_{1}, h_{1}'}(x) = \dfrac{1}{n}\sum_{k=1}^{n} \widetilde{U}_{k , 1}( x )$.
	We have   	
	\begin{align*} 
		\left| \widetilde{U}_{k,1} (x) \right|   
		&\leq \dfrac{1}{2\pi} \displaystyle{ \int_{\R} } \left|  \dfrac{  K_{h'_{1}}^\star(t)   K_{h_{1}}^\star(t) }{ f^{\star}_\varepsilon(t) } \right| dt  
	\leq    \dfrac{1}{2\pi}  \left\| K^{\star} \right\|_{\infty}     \left\| \dfrac{ K_{h_{1}}^\star }{ f^{\star}_\varepsilon } \right\|_{1},
	\end{align*}
	and $$\Var \big( \widetilde{U}_{k,1}(x) \big)  =  n  \Var\big(\widehat{p}_{1,h_{1},h_{1}'}(x) \big)  \leq  n  \left\| K^{\star} \right\|_{\infty} \widetilde V_{0}(n,h_{1}).$$\\
	Using similar arguments in proof of the first assertion on $\widehat{p}_{1, h_{1}}(x)$, we obtain with a probability greater than $1 - 2n^{-\alpha}$ that
	\begin{align*}
		\big| \widehat{p}_{1, h_{1} ,h_{1}'}(x) - \E \big[ \widehat{p}_{1, h_{1} ,h_{1}'}(x) \big] \big|    \leq   \tilde c_{1}(\alpha)   \left\| K^{\star} \right\|_{\infty}    \sqrt{\widetilde{V}(n,h_{1})} .
	\end{align*}
	This concludes the proof of Proposition~\ref{proposition:apply.Bernstein.inequa:study.adaptation}.
\end{proof}
\medskip

Now, we can start to prove Theorem~\ref{adaptive:prop:upper-bound.high-Proba:p1-p2}. As the proof is similar to the proof of Theorem  \ref{theorem-proba-oracle}, we will omit some computations.

Let $j \in \{1,2\}$ and $h_{j} \in \mathcal{H}$ be fixed. Recall that the target is to find an upper bound for $\big| \widehat{p}_{j,\widehat{h}_{j}}(x)  -  p_{j}(x) \big|$. We consider the following decomposition:
	\begin{align*}
		\big| \widehat{p}_{j,\widehat{h}_{j}}(x)  -  p_{j}(x) \big|   \leq   { \big| \widehat{p}_{j,\widehat{h}_{j}}(x)  -  \widehat{p}_{j, h_{j} ,\widehat{h}_{j}}(x) \big| } &+   {\big| \widehat{p}_{j, h_{j} ,\widehat{h}_{j}}(x)  -  \widehat{p}_{j, h_{j}}(x) \big|}  
		+  \hspace{0.1cm}   \big|  \widehat{p}_{j, h_{j}}(x)  -  p_{j}(x) \big| .
	\end{align*}
		Following proof of  Theorem \ref{theorem-proba-oracle}, we obtain
	\begin{equation} \label{adaptive:thm:upper-bound.high-Proba:p2:proof:inequality1}
		\big| \widehat{p}_{j,\widehat{h}_{j}}(x)  -  p_{j}(x) \big|   \leq   2  \widetilde A(h_{j} , x)  +   2 \sqrt{\widetilde{V}_{j}(n,h_{j})}  +   \big|  \widehat{p}_{j, h_{j}}(x)  -  p_{j}(x) \big| .
	\end{equation}
	For the  study of the term $\widetilde A(h_{j} , x)$, we recall that:
	$$
	\widetilde 	A(h_{j} , x) 
		= \sup_{h_{j}' \in \mathcal{H}}  \Big\{ \big| \widehat{p}_{j,h_{j}'}(x)  -  \widehat{p}_{j, h_{j} ,h_{j}'}(x) \big| -   \sqrt{\widetilde{V}_{j}(n,h_{j}')} \Big\}_{+} .
	$$

We get that
	\begin{align*}
		\big| \widehat{p}_{j,h_{j}'}(x)  -  \widehat{p}_{j, h_{j} ,h_{j}'}(x) \big| -   \sqrt{\widetilde{V}_{j}(n,h_{j}')}   
		\leq    &\hspace{0.2cm}  \big| \widehat{p}_{j,h_{j}'}(x)  -  \E \big[ \widehat{p}_{j,h_{j}'}(x) \big]  \big|  
		-    \dfrac{\sqrt{\widetilde{V}_{j}(n,h_{j}')}}{ (1 + \left\| K^{\star} \right\|_{\infty} ) }
		\\
		& \hspace{0.2cm}  +  \big| \widehat{p}_{j, h_{j} ,h_{j}'}(x) - \E \big[ \widehat{p}_{j, h_{j} ,h_{j}'}(x) \big] \big|  
		-  \left\| K^{\star} \right\|_{\infty}     \dfrac{\sqrt{\widetilde{V}_{j}(n,h_{j}')}}{ (1 + \left\| K^{\star} \right\|_{\infty} ) }
		\\[0.2cm]  
		& \hspace{0.2cm} +  \big| \E \big[ \widehat{p}_{j,h_{j}'}(x) \big]  -  \E \big[ \widehat{p}_{j, h_{j} ,h_{j}'}(x) \big] \big|.
	\end{align*}
	However, for any $h_{j}' \in \mathcal{H}$, with $M(K)$ 
	defined in~\eqref{equa:definition.M(K):form},  
	we get  
	\begin{align}\label{uniformbias}
		\big| \E \big[ \widehat{p}_{j,h_{j}'}(x) \big]  -  \E \big[ \widehat{p}_{j, h_{j} ,h_{j}'}(x) \big] \big|  &=  \big| K_{h_{j}'}* \big( p_{j} - K_{h_{j}}*p_{j} \big)(x) \big|  
		\leq  M(K)  \hspace{0.1cm}   
		\left\| p_{j} - K_{h_{j}} \ast  p_{j}   \right\|_{\infty}
		. 
	\end{align}
	Hence, 
	incorporating this bound in the definition of $\widetilde A(h_{j} , x )$, we obtain
	\begin{align} 
		\widetilde A(h_{j} , x) 
		&= \sup_{h_{j}' \in \mathcal{H}}  \Big\{ \big| \widehat{p}_{j,h_{j}'}(x)  -  \widehat{p}_{j, h_{j} ,h_{j}'}(x) \big| -   \sqrt{\widetilde{V}(n,h_{j}')} \Big\}_{+}   \nonumber
		\\
		&\leq  \sup_{h_{j}' \in \mathcal{H}}  \left\{\big| \widehat{p}_{j,h_{j}'}(x)  -  \E \big[ \widehat{p}_{j,h_{j}'}(x) \big]  \big|  -   \dfrac{\sqrt{\widetilde{V}(n,h_{j}')}}{ (1 +  \left\| K ^{\star} \right\|_{\infty}  ) }  \right\}_{+}
		\nonumber   
		\\
		&\hspace{0.4cm}  +  \sup_{h_{j}' \in \mathcal{H}}  \left\{\big| \widehat{p}_{j, h_{j} ,h_{j}'}(x) - \E \big[ \widehat{p}_{j, h_{j} ,h_{j}'}(x) \big] \big|  -   \left\| K ^{\star} \right\|_{\infty}    \dfrac{\sqrt{\widetilde{V}(n,h_{j}')}}{ (1 +  \left\| K ^{\star} \right\|_{\infty} ) }  \right\}_{+}
		\nonumber
		\\
		& \hspace{6cm} +   \left\| K ^{\star} \right\|_{\infty}  \left\|p_{j} - K_{h_{j}}*p_{j}\right\|_{\infty}    \nonumber
		\\[0.2cm]  
		&\leq  \sup_{h_{j}' \in \mathcal{H}}  \left\{\big| \widehat{p}_{j,h_{j}'}(x)  -  \E \big[ \widehat{p}_{j,h_{j}'}(x) \big]  \big|  -   \dfrac{\sqrt{\widetilde{V}(n,h_{j}')}}{ (1 +  \left\| K ^{\star} \right\|_{\infty}  ) }  \right\}_{+}   
		\label{adaptive:thm:upper-bound.high-Proba:p2:proof:inequality2}
		\\
		&\hspace{0.4cm}  +  \sup_{h_{j}' \in \mathcal{H}}  \left\{\big| \widehat{p}_{j, h_{j} ,h_{j}'}(x) - \E \big[ \widehat{p}_{j, h_{j} ,h_{j}'}(x) \big] \big|  -   \left\| K ^{\star} \right\|_{\infty}    \dfrac{\sqrt{\widetilde{V}(n,h_{j}')}}{ (1 +  \left\| K ^{\star} \right\|_{\infty} ) }  \right\}_{+}    
		\nonumber
		\\
		&\hspace{0.4cm}  +   M(K) 
		\left\| p_{j} - K_{h_{j}} \ast  p_{j}   \right\|_{\infty}.   
		\nonumber
	\end{align}
By  Proposition~\ref{proposition:apply.Bernstein.inequa:study.adaptation}
	
	\begin{align*}
		\mathbb{P} \left( \big| \widehat{p}_{j, h_{j} }(x) - \E \big[ \widehat{p}_{j, h_{j} }(x) \big] \big|  >  c_{1}(\alpha)    \sqrt{\widetilde{V}(n  ,  h_{j})} \right)  \leq   2 n^{-\alpha} .
	\end{align*}
	It implies that if we take $c_{1}(\alpha) = \dfrac{1}{1 + \left\| K^{\star} \right\|_{\infty}}$ and if  $c_{0,j} \geq 16 \alpha^{2} \big( 1 + \left\| K^{\star} \right\|_{\infty}  \big)^{2} \dfrac{1}{ \min  \left\{     \left\| f_{\varepsilon}^{\star} \right\|_{1}      ,     1  \right\} } $, then
	\begin{align*}
		\mathbb{P} \left(  \sup_{h_{j}' \in \mathcal{H}}  \left\{ \big| \widehat{p}_{j,h_{j}'}(x)  -  \E \big[ \widehat{p}_{j,h_{j}'}(x) \big]  \big|  -   \dfrac{\sqrt{\widetilde{V}(n,h_{j}')}}{ (1 + \left\| K^{\star}  \right\|_{\infty} ) }  \right\}_{+}  >  0  \right)  &\leq  2   \sum_{h_{j} \in \mathcal{H}} n^{-  \alpha}  \leq   2  n^{1 - \alpha}  ,
	\end{align*}
	as $\card(\mathcal{H})  \leq  n$ (see~\eqref{def:bandwidth.colletion.H_n}).
	Similarly, from  Proposition~\ref{proposition:apply.Bernstein.inequa:study.adaptation}, we can deduce that  
	\begin{align*}
		\mathbb{P} \left( \sup_{h_{j}' \in \mathcal{H}}  \left\{\big| \widehat{p}_{j, h_{j} ,h_{j}'}(x) - \E \big[ \widehat{p}_{j, h_{j} ,h_{j}'}(x) \big] \big|  -   \left\| K ^{\star} \right\|_{\infty}    \dfrac{\sqrt{\widetilde{V}_{j}(n,h_{j}')}}{ (1 +  \left\| K ^{\star} \right\|_{\infty} ) }  \right\}_{+}   >   0   \right)  \leq   2 n^{1-\alpha} .
	\end{align*}
	Consequently, the following set
	\begin{align*}
		\widetilde {\mathcal {E}_j }
		&:=  \Bigg\{  \sup_{h_{j}' \in \mathcal{H}}  \Big\{ \big| \widehat{p}_{j,h_{j}'}(x)  -  \E \big[ \widehat{p}_{j,h_{j}'}(x) \big]  \big|  -    \dfrac{\sqrt{\widetilde{V}_{j}(n,h_{j}')}}{ (1 +  \left\| K^{\star} \right\|_{\infty} ) }  \Big\}_{+}  = 0  \Bigg\}
		\\
		&\hspace{0.5cm}  \bigcap  \Bigg\{ \forall \, h_j \in \mathcal{H},    \sup_{ h_{j}' \in \mathcal{H} }  \Big\{ \big| \widehat{p}_{j, h_{j} ,  h_{j}'}(x) - \E \big[ \widehat{p}_{j, h_{j} ,  h_{j}'}(x) \big] \big|    
		-   \left\| K^{\star} \right\|_{\infty}   \dfrac{\sqrt{\widetilde{V}_{j}(n,h_{j}')}}{ (1 + \left\| K^{\star} \right\|_{\infty} ) }  \Big\}_{+} = 0  \Bigg\}
	\end{align*}
	has probability greater than $1 - 4 \, n^{2 - \alpha}$. 
	\leavevmode \\
	Now, for $\alpha  \geq  3$, choose $\alpha = 2 + q$ and then $\tilde c_{0,j} \geq 16 \big(2 + q \big)^{2} \big( 1 + \left\| K^{\star} \right\|_{\infty} \big)^{2}  \dfrac{1}{ \min  \left\{     \left\| f_{\varepsilon}^{\star} \right\|_{1}      , 1  \right\} }  $. Thus, we obtain that $\mathbb{P} \big( \widetilde {\mathcal {E}_{j}  } \big) > 1 - 4  \,  n^{-q} $.
	Combining  \eqref{adaptive:thm:upper-bound.high-Proba:p2:proof:inequality1} and \eqref{adaptive:thm:upper-bound.high-Proba:p2:proof:inequality2}, we have on $\widetilde {\mathcal {E}_j  }$:
	\begin{align*}
		\big| \widehat{p}_{j,\widehat{h}_{j}}(x)  -  p_{j}(x) \big|   &\leq   2    \widetilde A(h_{j},x)   +   2    \sqrt{\widetilde{V}_{j}(n,h_{j})}  +   \big|  \widehat{p}_{j, h_{j}}(x)   -   p_{j}(x) \big|    
		\\
		&\leq 2 M(K) \, 
		\left\| p_{j} - K_{h_{j}} \ast  p_{j}   \right\|_{\infty}    
		+  2   \sqrt{\widetilde{V}_{j}(n,h_{j})}  
		+   \big|  \widehat{p}_{j, h_{j}}(x) - p_{j}(x) \big| ,
	\end{align*}
	but still on $\widetilde {\mathcal {E}_j  }$, one gets $\big| \widehat{p}_{j , h_{j}}(x) - \E \big[\widehat{p}_{j , h_{j}}(x)\big] \big|  -   \dfrac{\sqrt{\widetilde{V}_{j}(n , h_{j})}}{ (1 + \left\| K^{\star} \right\|_{\infty } ) } \leq 0$, so
	\begin{align*}
		\big| \widehat{p}_{j,h_{j}}(x)   -   p_{j}(x) \big|   &\leq  \big| \E \big[\widehat{p}_{j,h_{j}}(x)\big]   -    p_{j}(x) \big|  +  \big| \widehat{p}_{j,h_{j}}(x) - \E \big[\widehat{p}_{j,h_{j}}(x)\big] \big|  
		\\ & \hspace{2.5cm}
		-  \dfrac{\sqrt{\widetilde{V}_{j}(n,h_{j})}}{ (1 + \left\| K^{\star}  \right\|_{\infty} ) } 
		+  \dfrac{\sqrt{\widetilde{V}_{j}(n,h_{j})}}{ (1 + \left\| K^{\star} \right\|_{\infty} ) }
		\\
		&\leq   
		\left\| p_{j} - K_{h_{j}} \ast  p_{j}   \right\|_{\infty}
		+    \sqrt{\widetilde{V}_{j}(n,h_{j})}.   
	\end{align*}
	Therefore, on $\widetilde {\mathcal {E}_j }$, we finally obtain
	\begin{align*}
		\big| \widehat{p}_{j,\widehat{h}_{j}}(x)   -    p_{j}(x) \big|   &\leq   
		\big( 1 + 2 M(K) \big) \, 
		\left\| p_{j} - K_{h_{j}} \ast  p_{j}   \right\|_{\infty}
		+  3  \sqrt{\widetilde{V}_{j}(n,h_{j})} .
	\end{align*}
	This concludes the proof of Theorem~\ref{adaptive:prop:upper-bound.high-Proba:p1-p2}.

%
%
\subsubsection{Proof of Theorem~\ref{adaptive:thm:pointwise-risk.upper-bound:m} }
\label{sec:adaptive:thm:pointwise-risk.upper-bound:m:proof}  
\noindent  
First, we establish a concentration result for    $\widehat{p}_{j,\widehat{h}_{j}}(x)$ in the following lemma. 
\begin{lem} \label{adaptive:thm:pointwise-risk.upper-bound:p1-p2}
Suppose that the noise $\varepsilon$ is OS  with $r > 1$.  Let $j \in \{ 1,  2 \} $ and  $p_{j}$ belong to $\mathcal{H}(\beta_{j}, \Lambda_{j})$. Assume that the kernel K satisfies Assumption~\ref{assumption:K^star:ordinary.smooth.noise:form} and  Assumption~\ref{assumption:kernel.K} with an index $\kappa \in \mathbb{R}_{+}$ such that $\kappa \geq  \beta_{j}$. 
Then under assumptions of Theorem~\ref{adaptive:prop:upper-bound.high-Proba:p1-p2},  there exist constants $C_{j}$  depending on $\beta_{j}, \Lambda_{j}, \tilde c_{0,j}, K$ and~$f_{\varepsilon}^{\star}$ such that, with  
	\begin{align*}
		\widetilde{E}_{j}  :=\left\{\big| \widehat{p}_{j,\widehat{h}_j}(x) - p_{j}(x) \big|  \leq C_{j}  \,  \Big( \dfrac{ \log(n) }{n} \Big)^{\frac{2 \beta_{j}}{2 \beta_{j} + 2 r + 1  }}   \right\}  ,  
	\end{align*}
	we have, for $n$ large enough, 
	\begin{align*}
		\mathbb{P} \Big( 	\widetilde{E}_{j}  \Big)  > 1 -  4 n^{-q} . 
	\end{align*}
\end{lem}
\begin{proof}
Let  $j\in\{1,2\}$.
	Since $p_{j}$ belongs to $\mathcal{H}(\beta_{j}, \Lambda_{j})$, from  Proposition~\ref{lem:point-wise:mean-var.p1-p2},  we have
	\begin{align*}
		\Big| \mathbb{E}\big(\widehat{p}_{j , h_{j} }(x)\big) - p_{j}(x) \Big|  
		\leq     C_{K,\mathcal{\kappa}} \,  \Lambda_{j} \,  h_{j}^{\beta_{j}}  ,    
	\end{align*}
and
	\begin{align*}
		\widetilde V(n,h_{j})  \precsim    \dfrac{    \log n  }{n} h_{j}^{-1 - 2\r}  .  
	\end{align*}
	Thus, we define $h_{j,OS}$ by    
	\begin{align}
		h_{j,OS} := \underset{h_{j} \in \left\{k^{-1} :\, k \in \N^{*}  \right\} 
		}{\textrm{arg min}}  \left\{   h_{j}^{2 \beta_{j}}  +    \frac{\log n }{n} h_{j}^{-1 - 2\r}     \right\} .  
		\label{adaptive:corollary:concentration:p1-p2:proof.bound.g1}
	\end{align}
	Equation \eqref{adaptive:corollary:concentration:p1-p2:proof.bound.g1} yields that $h_{j,OS} $ is of order $\Big( \dfrac{\log(n)}{n} \Big)^{\frac{1}{2 \beta_{j} + 1 + 2r }}  $. Let us show that $h_{j,OS} \in  \mathcal{H}$. 

Let $0 \leq h \leq 1$. We have (using similar computations as in \cite{Comte-Lacour} page 598)
	\begin{align*}
		\left\| \dfrac{K^{\star}_{h} }{ f_{\varepsilon}^{\star} }   \right\|_{2}^{2}   \Bigg( \left\| \dfrac{K^{\star}_{h} }{ f_{\varepsilon}^{\star} }   \right\|_{1} \Bigg)^{-2}    
		&=    \int_{\R}  \Big| \dfrac{K^{\star}(t \, h_{}) }{ f_{\varepsilon}^{\star}(t) } \Big|^{2} dt   \hspace{0.1cm}  \Bigg( \int_{\R}  \Big| \dfrac{K^{\star}(t \, h_{}) }{ f_{\varepsilon}^{\star}(t) } \Big| dt    \Bigg)^{-2}
		\\
		&\geq  h_{}  \,   c_{\varepsilon}^{-2}   C_{\varepsilon}^{-2}    \int_{\R}  \big| K^{\star}(u) \big|^{2}  \,  \big| u \big|^{2r}   du   \,  \Bigg( \int_{\R}  \big| K^{\star}(v) \big|  \Big( 1 + \big| v \big| \Big)^{r} dv    \Bigg)^{-2}.  
				\end{align*}
				Consequently, 
				\begin{align}
		\left\| \dfrac{K^{\star}_{h_{j,OS}} }{ f_{\varepsilon}^{\star} }   \right\|_{2}^{2}   \Bigg( \left\| \dfrac{K^{\star}_{h_{j,OS}} }{ f_{\varepsilon}^{\star} }   \right\|_{1} \Bigg)^{-2}    		& \geq C \Big( \dfrac{\log(n)}{n} \Big)^{\frac{1}{2 \beta_{j} + 1 + 2r}}  \,   c_{\varepsilon}^{-2}   C_{\varepsilon}^{-2}    \int_{\R}  \big| K^{\star}(u) \big|^{2}  \,  \big| u \big|^{2r}   du   \,  \Bigg( \int_{\R}  \big| K^{\star}(v) \big|  \Big( 1 + \big| v \big| \Big)^{r} dv    \Bigg)^{-2}    
		\nonumber  
		\\
		&\geq  \dfrac{\log(n)}{n}
		, 
		\label{proof:adaptive:corollary:pointwise-risk.upper-bound:p1-p2:h_optimal.in.Hn:verification:ordinary.noise.case}
	\end{align}
	since $\dfrac{1}{2\beta_{j} + 1 + 2 r} < 1$ for $\beta_{j} > 0, r > 0$ and 
	Assumption~\ref{assumption:K^star:ordinary.smooth.noise:form} guarantees that all integrals finite. Thus, $h_{j, OS} \in \mathcal{H}$. 
	As a result, using Theorem~\ref{adaptive:prop:upper-bound.high-Proba:p1-p2}, we obtain 
	with probability greater than $1 - 4 \,  n^{-q}$, that
	\begin{align*}
		\big| \widehat{p}_{j,\widehat{h}_j}(x) - p_{j}(x) \big|   \leq  
		C_{j} \,  
		\Big( \dfrac{\log(n)}{n} \Big)^{\frac{ \beta_{j} }{2 \beta_{j} + 1 + 2r }}  \,   ,    
	\end{align*}  
	with  
	constant $C_{j}$ depending on $\beta_{j}, \Lambda_{j}, \tilde c_{0,j}$, $K$ and  $f_{\varepsilon}^{\star}$. 
	This concludes the proof of Lemma~\ref{adaptive:thm:pointwise-risk.upper-bound:p1-p2}.   
\end{proof}

\bigskip  

Let us now turn to the proof of Theorem \ref{adaptive:thm:pointwise-risk.upper-bound:m}.
 It is very similar to the proof of Theorem   \ref{adaptive:thm:circular} except that one has to use Lemma \ref{adaptive:thm:pointwise-risk.upper-bound:p1-p2} instead of Lemma \ref{proba-hat-p}. Consequently, we omit the proof.

\end{document}